\documentclass[11pt,letterpaper]{amsart}
\usepackage{hyperref}
\usepackage{graphicx, color}
\usepackage{hyperref}
\usepackage{mathtools}
\usepackage{enumerate}
\usepackage{bbm}
\usepackage[margin=1in]{geometry}
\usepackage{cool}

\newtheorem{theorem}{Theorem}[section]
\newtheorem{proposition}[theorem]{Proposition}
\newtheorem{lemma}[theorem]{Lemma}
\newtheorem{corollary}[theorem]{Corollary}
\theoremstyle{definition}
\newtheorem{definition}[theorem]{Definition}

\theoremstyle{remark}

\let\Re\undefine
\DeclareMathOperator{\Re}{Re}
\DeclareMathOperator{\sign}{sgn}
\DeclareMathOperator{\tr}{tr}

\DeclareMathOperator{\diag}{diag}

\newcommand{\tp }{{\scriptscriptstyle\mathsf{T}}}

\newcommand{\arccosh}{\textrm{arccosh\,}}

\newcommand{\G}{\mathsf{G}}

\begin{document}
\title{Haagerup bound for quaternionic
Grothendieck inequality}
\author{Shmuel~Friedland}
\address{Department of Mathematics, Statistics and Computer Science,  University of Illinois, Chicago, IL,  60607-7045.}
\email{friedlan@uic.edu}
\author{Zehua~Lai}
\address{Computational and Applied Mathematics Initiative, Department of Statistics,
University of Chicago, Chicago, IL 60637-1514.}
\email{laizehua@uchicago.edu}
\author{Lek-Heng~Lim}
\address{Computational and Applied Mathematics Initiative, Department of Statistics,
University of Chicago, Chicago, IL 60637-1514.}
\email[corresponding author]{lekheng@galton.uchicago.edu}
\date{April 5, 2020}

\begin{abstract}
We present here several versions of the Grothendieck inequality over the skew field of quaternions: The first one is the standard  Grothendieck inequality for rectangular matrices, and two additional inequalities for self-adjoint matrices, as introduced by the first and the last authors in a recent paper.  We give several results on ``conic Grothendieck inequality'': as Nesterov $\pi/2$-Theorem, which corresponds to the cones of positive semidefinite matrices;  the  Goemans--Williamson inequality, which corresponds to the cones of weighted Laplacians; the diagonally dominant matrices. 
The most challenging technical part of this paper is the proof of the analog of Haagerup result that the inverse of the hypergeometric function $x\; \Hypergeometric{2}{1}{\frac{1}{2}, \frac{1}{2}}{3}{x^2}$
has first positive Taylor coefficient and all other Taylor coefficients are nonpositive.
\end{abstract}

% tensor norm}

\maketitle

\maketitle

\section{Introduction}\label{sec:intro}

We will let $\mathbb{F} = \mathbb{R}$ and $\mathbb{C}$ and $\mathbb{H}$  be the fields of real, complex and the skew field of quaternions  respectively in this article. In $1953$, Grothendieck proved a powerful result that he called ``the fundamental theorem in the metric theory of tensor products''  \cite{Grothendieck}. His result can be stated as follows \cite{Lindenstrauss}:  For $\mathbb{F}\in\{\mathbb{R},\mathbb{C}\}$ there exists a finite  constant $K > 0$ such that for every $l,m,n\in\mathbb{N}$ and every matrix $M=(M_{ij})\in\mathbb{F}^{m\times n}$, 
\begin{equation}\label{GI}
\max_{\lVert x_i\rVert = \lVert y_j \rVert = 1}\biggl| \sum_{i=1}^m\sum_{j=1}^n M_{ij} \langle x_i, y_j\rangle\biggr|\leq K \max_{\lvert \varepsilon_i \rvert = \lvert \delta_j \rvert = 1}\biggl|\sum_{i=1}^m\sum_{j=1}^n M_{ij}  \bar \varepsilon_i\delta_j\biggr|
\end{equation}
where the maximum on the left is take over all $x_i,y_j\in\mathbb{F}^{l}$ of unit $2$-norm, and the maximum on the right is taken over all  $\varepsilon_i, \delta_j \in \mathbb{F}$ of unit absolute value (i.e., $\varepsilon_i = \pm 1$,  $\delta_j=\pm1$ over $\mathbb{R}$; $\varepsilon_i = e^{i\theta_i}$, $\delta_j= e^{i \phi_j}$ over $\mathbb{C}$).  The inequality \eqref{GI} has since been christened \emph{Grothendieck's inequality} and the smallest possible constant $K$ \emph{Grothendieck's constant}. The value of Grothendieck's constant depends on the choice of $\mathbb{F}$ and we will denote it by $K_G^\mathbb{F}$.  In a recent paper \cite{FL20} two authors of this paper extended
the Grothendieck inequality to symmetric/Hermitian matrices, which we call symmetric Grothendieck inequality and referred as SGI.  Namely, in the above inequality we can assume that $M$ is symmetric/Hermitian and $x_i=y_i$.  
Furthermore, they considered more refined versions of SGI where the vectors $x_i$ are in $d$-dimensional Hilbert space as in \cite{Bri1}.

The aim of this paper to extend the Grothendieck's inequality and SGI to quaternions $\mathbb{H}$.  Since quaternions is a skew-field, which is noncommutative, there are a number of obstructions to overcome, to have the Haagerup type constant \cite{Haagerup2}.  We now describe briefly the results we obtained.  Let $\mathbb{F}=\mathbb{H}$.  We first show that the inequality \eqref{GI} holds, where $x_i,y_j$ in the quaternion Hilbert space $\mathbb{H}^l$, where $l\ge m+n$, and $\varepsilon_i,\delta_j\in \mathbb{H}$ with the constant $K_\G^\mathbb{H}$.  Using the analogous results to Krivine \cite{Krivine2} and Haagerup \cite{Haagerup2} we show that   $K_\G^\mathbb{H}\le 1.2168$.  This result is achieved by establishing the most difficult technical part of our paper.  Let $\Hypergeometric{2}{1}{a, b}{c}{x}$ be the classical hypergeometric function.  Denote $p_\ell(x)=x\; \Hypergeometric{2}{1}{\frac{1}{2}, \frac{1}{2}}{\ell}{x^2}$ for $\ell\in \mathbb{N}$.  It was shown by Haagerup that the inverse function $p^{-1}_\ell(x)$ has first positive Taylor coefficient, while all other Taylor coefficients are nonpositive for $\ell=2$.  In this paper we show Haagerup result for $\ell=3$.  Numerical computations show that the same result holds for at least  $\ell= 4,5,6$, for the first one hundred Taylor coefficients.

Denote by $\mathbb{S}^n(\mathbb{F})\subset \mathbb{F}^{n\times n}$ the real space of self-adjoint matrices. i.e., $A^*=A$.  We show that we have two analogs of the Grothendieck inequaity \eqref{GI} on $\mathbb{S}^n(\mathbb{H})$:
\begin{eqnarray}\label{Kgamineq}
&&\max _{\|x_i\|=1}\biggl|\Re \sum_{i=1}^n \sum_{j=1}^n a_{ij}\langle x_i, x_j \rangle\biggr|\le K_\gamma^\mathbb{H}
 \max _{|\delta_i|=1}\biggl| \Re\sum_{i=1}^n \sum_{j=1}^n  a_{ij}\bar \delta_i\delta_j  \biggr|,\\
&&\max _{\|x_i\|\le1}\biggl|\Re \sum_{i=1}^n \sum_{j=1}^n a_{ij}\langle x_i, x_j \rangle\biggr|\le K_\Gamma^\mathbb{H}
 \max _{|\delta_i|\le1}\biggl| \Re\sum_{i=1}^n \sum_{j=1}^n  a_{ij}\bar \delta_i\delta_j  \biggr|.
 \notag
\end{eqnarray}
Furthermore $K_\G^\mathbb{H}\le K_\Gamma^\mathbb{H}\le  K_\gamma^\mathbb{H}\le \frac{64}{9\pi}-1\approx 1.263537$.

We now describe briefly  the ``conic Grothendieck inequality'' for various cones in $\mathbb{S}^n(\mathbb{H})$.  Denote by $\mathbb{S}^n_+(\mathbb{H})$ the cone of positive semdefinite self-adjoint quaternionic matrices.  We show that in this case \eqref{GI} is equivalent to the inequality of the form \eqref{Kgamineq} with the constant  $32/9\pi$, which is sharp.   This is a quaternionic version of Nesterov-Rietz $\pi/2$ theorem \cite{Nes,Rietz} for the real numbers, and Nemirovski-Roos-Terlaky $4/\pi$ theorem for the complex numbers \cite{NRT99}.  We next consider the subcone of $\mathbb{S}^n_+(\mathbb{R})$ of Laplacian matrices.  In this case the constant in \eqref{GI} can be reduced to $K\leq 1.0338$.  This is the quaternion version of the celebrated Goemans-Williamson inequality \cite{GW95}.

\section{Quaternions}\label{sec:quat}
\subsection{Basic facts on quaternions}\label{subsec:quat}
Recall that $\mathbb{H}$ can be viewed as $\mathbb{R}^4$.  So $a\in\mathbb{H}$ is of the form $a=a_0+a_1\mathrm{i}+a_2\mathrm{j}+a_3\mathrm{k}$.  We can identify $a$ with $a=(a_0,a_1,a_2,a_3)^{\top}$. We define the real part of $a$ to be $\Re a:=a_0$, and the conjugate of $a$ to be $ \bar a=a^*=a_0-a_1\mathrm{i}-a_2\mathrm{j}-a_3\mathrm{k}$.
The product table of
$\mathrm{i},\mathrm{j},\mathrm{k}$ is given by 
\[\mathrm{i}^2=\mathrm{j}^2=\mathrm{k}^2=-1,\;
\mathrm{i}\mathrm{j}=-\mathrm{j}\mathrm{i}=\mathrm{k},\; \mathrm{j}\mathrm{k}=-\mathrm{k}\mathrm{j}=\mathrm{i},\;
\mathrm{k}\mathrm{i}=-\mathrm{i}\mathrm{k}=\mathrm{j}.\]
Hence $\mathbb{H}$ is a noncommutative ring over $\mathbb{R}$.  Observe next that
$a \bar a=\bar a a=a_0^2+a_1^2+a_2^2+a_3^2$.  Hence $|a|=\sqrt{a \bar a}\ge 0$ and equality holds if and only $a=0$.  Thus for $a\ne 0$ the element $|a|^{-1} \bar a=\bar a|a|^{-1}$ is the unique inverse of $a$ in $\mathbb{H}$.  So $\mathbb{H}$ is a skew field over the field $\mathbb{R}$, where $1$ is the identity element.  Frobenius's theorem claims that the only skew fields over $\mathbb{R}$ are $\mathbb{R},\mathbb{C},\mathbb{H}$.
For $A\in \mathbb{F}^{m\times n}$ we denote $A^*:=\bar A^\tp \in \mathbb{F}^{n\times m}$.

There is standard way to to present quaternions similar to the complex numbers:
$z+w\mathrm{j}$, where $z,w\in \mathbb{C}$.  Indeed, if $z=x+y\mathrm{i}, w=u+v\mathrm{i}$, the the identity $\mathrm{i}\mathrm{j}=\mathrm{k}$ yields
$z+w\mathrm{j}=x+y\mathrm{i}+u\mathrm{j}+v\mathrm{k}$.  Thus to multiply quaternions we have to remember that the product of complex numbers is commutative and
\begin{equation}\label{quaterrues}
\overline {z+w\mathrm{j}}=\bar z-w\mathrm{j}, \quad w\mathrm{j}=\mathrm{j}\bar w.
\end{equation}

There is another representation of $\mathbb{H}$ as a real subalgebra of $2\times 2$ complex valued matrices $\mathbb{C}^{2\times 2}$.  First observe that one can view $a$ as $a=(z,w)\in\mathbb{C}^2$.  Note that $\bar a=(\bar z, -w)$.  (Warning: if one views $(z,w)$ as a vector with complex entires then $\overline{(z,w)}=(\bar z, \bar w)$.)  Let 
\begin{eqnarray}\label{Carep}
C(a)=\left(\begin{array}{cc}z&w\\-\bar w&\bar z\end{array}\right)\in \mathbb{C}^{2\times 2}, \quad a=(z,w).
\end{eqnarray}
Then the map $a\to C(a)$ is an isomorphism of $\mathbb{H}$ and the induced complex $2$-dimensional subalgebra $\mathcal{C}(\mathbb{H})=\{C(a), \;a\in\mathbb{H}\}\subset\mathbb{C}^{2\times 2}$.   Note that ${A}(\mathbb{H})\cap \mathbb{R}^{2\times 2}$ is subalgebra isomorphic to $\mathbb{C}$.
Observe that 
\begin{equation}\label{tracequat}
|a|^2=\det C(a), \quad C(\bar a)=C(a)^*, \Re(a)=\frac{1}{2}\tr(C(a)).
\end{equation} 
As $\tr(AB)=\tr(BA)$ we deduce that
\begin{equation}\label{tracecomrel}
\Re (ab)=\Re( ba)=\Re(\overline{ab})=\Re (\bar b \bar a)=\Re(\bar a \bar b),\,a,b\in\mathbb{H}.
\end{equation}
\subsection{Vector spaces}\label{subsec:vecsp}
We next consider a right vector space $\mathbb{V}$ over $\mathbb{H}$.  It is a commutative group with $0$  element denoted as $0$. We will denote in this section by the lower case bold letter vectors in $\mathbb{V}$. For the right vector space $\mathbb{V}$ the scalar vector product $va$ satisfies the standard assumptions:
\[(v+w)a=va+wa, \; v(a+b)=va+wb,\; v(ab)=(va)b,\; v1=v.\]
We can define similarly the left vector space over $\mathbb{H}$. In this paper, we only work with right vector space.
Linear dependence, linear independence, subspace, span of a set of vectors, finitely generated subspaces, basis are defined as for the vector spaces over a field.  Every finitely generated vector space over $\mathbb{H}$ has a basis of the same cardinality, which is denoted by $\textrm{dim }\mathbb{V}$.  Denote $[l]=\{1,\dots, l\}\subset \mathbb{N}$.  We view
\begin{eqnarray*}
\mathbb{H}^l=\{x=(x_1,\dots,x_l)^\tp, \;x_i\in\mathbb{H}, i\in[l]\}, \quad 
\mathbb{H}_l=\{x=(x_1,\dots,x_l), \;x_i\in\mathbb{H}, i\in[l]\},
\end{eqnarray*}
as right and left vector spaces over $\mathbb{H}$ respectively. 
Clearly, dim$\,\mathbb{H}^l=\textrm{dim}\,\mathbb{H}_l=l$ and $e_i=(\delta_{1i},\dots,\delta_{li}), i\in[l]$ is the standard basis in $\mathbb{H}^l$ and $\mathbb{H}_l$.

When a basis is specified, for example, the standard basis in the right vector space $\mathbb{H}^l$, the expression $av$ is meaningful and we will use it when necessary. However, the reader should keep in mind this expression should be treated as an additional structure related to a particular basis. Its meaning will be different if we take a different basis.  

Denote 
\begin{equation}\label{defbarstarvector}
\overline{x}=(\bar x_1,\dots,\bar x_l)^\tp, \, x^*=(\bar x_1,\dots,\bar x_l), \quad \textrm{for } x=(x_1,\dots,x_l)^\tp\in\mathbb{H}^l.
\end{equation}
Similar notations apply for $x\in\mathbb{H}_l$.
Note that $\overline{x}$ is defined with respect to the standard basis in $\mathbb{H}_l$.

Let $M=(M_{ij})\in\mathbb{H}^{m\times n}$. Define $C(M)=(C(M_{ij}))\in \mathbb{C}^{(2m)\times (2n)}$ to be the block matrix with $2\times 2$ blocks $C(M_{ij})$. Again, this embedding commutes with conjugate transpose, addition and multplication of matrices. For $m=n$ the matrix $M\in \mathbb{H}^{n\times n}$ is called (quaternion) self-adjoint if $M^*=M$.  We denote by $\mathbb{S}^n(\mathbb{F})\subset \mathbb{F}^{n\times n}$ the real space of self-adjoint matrices: $M^*=M$.  When no ambiguity arises we will drop the dependence on $\mathbb{F}$.

It is helpful to introduce a convenient relabeling of the block matrix $C(M)$ denoted as $\hat C(M)=P_mC(M)P_{n}^\tp$, where $P_m\in \{0,1\}^{(2m)\times (2m)}$ is the following permutation matrix: The matrix $P_m$  permutes the rows $1,2,\dots, m+1,\dots,2m$ to $1,3,\dots,2m-1,2,4,\dots,2m$ respectively.  Then $\hat C(M)$ has the following block structure:
\begin{eqnarray}\label{hatCMrep}
\hat C(M)=\left(\begin{array}{cc} Z& W\\-\overline{W}&\bar Z\end{array}\right), \quad Z,W.\in \mathbb{C}^{m\times n}
\end{eqnarray}
Clearly, this partition is another isomorphism $\iota:\mathbb{H}^{m\times n}\to \mathbb{C}^{(2m)\times (2n)}$ which is preserved under multiplication and conjugate transpose of matrices.  Note $M\in \mathbb{S}^n(H)$ if in the above representation of $\hat C(M)$, where $Z\in \mathbb{S}^n(\mathbb{C})$ and $W$ is skew symmetric $W^\tp=-W$.
\subsection{Inner product on quaternion vector space}
Assume that $\mathbb{V}$ is a right vector space over $\mathbb{H}$.  
A mapping $\langle\cdot,\cdot\rangle: \mathbb{V}\times\mathbb{V}\to \mathbb{H}$ is called an inner product if the following conditions hold:
\begin{eqnarray*}
\langle y,x\rangle=\overline{\langle x,y\rangle},\\
\langle xa+yb,z\rangle=\bar a\langle x,z\rangle
+\bar b\langle y,z\rangle,\\
\langle z,xa+yb\rangle=\langle z,x\rangle a+
\langle z,y\rangle b,\\
\langle x,x\rangle >0 \textrm{ for } x\ne 0.
\end{eqnarray*}
The norm is defined as $\|x\|=\sqrt{\langle x,x\rangle}$. Let $\mathcal{H}$ be a right vector space of $\mathbb{H}$ with an inner product. We also call $\mathcal{H}$ the Hilbert space over quaternions. All analysis in this paper is essentially finite dimensional. Whether or not we incooperate the completeness in our definition does not change our result.
\begin{lemma}
\begin{enumerate}
\item $\|xa\|=|a|\|x\|$.
\item The Cauchy-Schwarz inequality holds for quaternion vector space,
\[
|\langle x,y\rangle| \leq \|x\| \|y\|.
\]
\item $\|\cdot\|$ is subadditive, i.e., satisfies the triangle inequality.  Hence $\|\cdot\|$ is indeed a norm on $\mathbb{V}$.
\end{enumerate}
\end{lemma}
\begin{proof}
\begin{enumerate}
\item $\|xa\|^2=\langle xa,xa\rangle = a^* \langle x,x\rangle a =  \|x\|^2 |a|^2$.
\item Suppose that $x$ is not a scalar multiple of $y$, and that neither $x$ nor $y$ is 0. Then $x-ya$ is not 0 for any $a$. So
\[
\|x-ya\|^2 = \|x\|^2+\|y\|^2|a|^2-\langle x,y\rangle a - a^* \langle x,y\rangle > 0
\]
Let $a = t\mu$ with real $t$ and $|\mu| = 1$ so that $\langle x,y\rangle a = |\langle x,y\rangle |t$. Then
\[
\|x\|^2+\|y\|^2t^2-2|\langle x,y\rangle |t> 0
\]
holds for all $t$. So $|\langle x,y\rangle| \leq \|x\| \|y\|.$
\item $ (\|x\|+\|y\|)^2 - \|x+y\|^2 = 2\|x\| \|y\|-\langle x,y\rangle-\langle y,x\rangle \geq 0. $
\end{enumerate}
\end{proof}

Two vectors $x,y$ are called orthogonal if $\langle x,y\rangle =0$.  A set of vectors $x_1,\dots,x_l\in \mathbb{V}$ is an orthonormal system if $\langle x_i,x_j\rangle=\delta_{ij}$ for $i,j\in[l]$.
\begin{lemma}\label{GSP} (Gram-Schmidt process)  Let $x_1,\dots,x_n$ be vectors in a right inner product space $\mathbb{V}$ over $\mathbb{H}$.
Assume that $x_1\ne 0$.  Then there exists $m\in[n]$ orthonormal vectors $y_1,\dots,y_m$ in the span of $x_1,\dots,x_n$ with the following property.  For each $i\in[n]$ there exists $j(i)\in[i]$ such that $x_1,\dots,x_i$ are in the span of $y_1,\dots,y_{j(i)}$.
The vectors $y_1,\dots,y_m$ are obtained by the Gram-Schmidt process.
\end{lemma}
\begin{proof}  Let $y_1=x_1\|x_1\|^{-1}$.  Suppose we defined the orthonormal vectors $y_1,\dots,y_j$ such that their span, denoted as $\mathbb{V}_j$,
contains the vectors $x_1,\dots,x_i$.  So $j(i)=j$.  Let
\[z_{i+1}=x_{i+1}-\sum_{k=1}^j  y_k \langle y_k,x_{i+1}\rangle.\]
Assume first that $z_{i+1}\ne 0$.  A straightforward calculation shows that  $z_{i+1}$ is orthogonal on $y_k$ for $k\in [j]$.
Then let $y_{j+1}=z_{i+1} \|z_{i+1}\|^{-1}$.
Assume second that $z_{i+1}=0$.  Then $j(i+1)=j$ and  we replace $x_{i+1}$ by $x_{i+2}$.
\end{proof}

\subsection{Tensor products over quaternions}\label{subsec:tneprodquat}
There is no natural way to define the tensor product space over  $\mathbb{H}$. So the definition below is coordinate dependent and should not be confused with the universal construction often used in other settings. 

Given quaternion vector spaces $\mathbb{H}^m, \mathbb{H}^n$ with standard basis, we define the tensor product $\mathbb{H}^m\otimes \mathbb{H}^n$ as the space of $\mathbb{H}^{m\times n}$ matrices and $u\otimes v$ can be identified with $uv^\tp$. On matrices $\mathbb{H}^m\otimes \mathbb{H}^n$ we define the inner product as:
\[\langle A,B \rangle=\textrm{ trace } A^*B=\sum_{i,j=1}^{m,n} A_{ij}^*B_{ij}, \; A=(A_{ij}),B=(B_{ij})\in\mathbb{H}^m\otimes \mathbb{H}^n.\]
  
Therefore 
\[
\langle u\otimes v,x\otimes y\rangle=
\sum_{i,j=1}^{m,n}v_j^* u_i^* x_i y_j.
\]
Notice that 
\[
\langle u, x\rangle \langle v, y\rangle=\sum_{i,j=1}^{m,n}u_i^* x_i v_j^*y_j.
\]
In general, the two quantities are not the same due to noncommutativity.

\begin{lemma}\label{teninprodform}
Let $u, x\in \mathbb{H}^m, v, y\in \mathbb{H}^n$.
\begin{enumerate}
\item If $\langle u,x\rangle\in \mathbb{R}$ then $\langle u\otimes v,x\otimes y\rangle=
\langle u,x\rangle \langle v,y\rangle$.
\item  $\|u\otimes v\|=\|u\|\|v\|$.
\item
$\Re(\langle \overline{v}\otimes v,\overline{y}\otimes y\rangle)=| \langle v,y\rangle|^2.$
 \end{enumerate}
\end{lemma}
\begin{proof}
\emph{(1)} If $\langle u,x\rangle\in \mathbb{R}$, then
\[
\sum_{i,j=1}^{m,n}v_j^* u_i^* x_i y_j = \sum_{j=1}^{n}v_j^*(\sum_{i=1}^{m} u_i^* x_i) y_j = \sum_{i,j=1}^{m,n}u_i^* x_i v_j^*y_j.
\]
\noindent
\emph{(2)} As $\langle u,u\rangle\ge 0$ it follows that$\langle u\otimes v,u\otimes v\rangle=
\langle u,u\rangle \langle v,v\rangle$.  Hence  $\|u\otimes v\|=\|u\|\|v\|$.

\noindent
\emph{(3)}
\[
\Re(\langle \overline{v}\otimes v,\overline{y}\otimes y\rangle)=\Re(\sum_{i,j=1}^{n,n}v_j^* v_i y_i^* y_j)=\Re(\sum_{i=1}^{n} v_i y_i^* \sum_{i=1}^{n}y_i v_i^*)=| \langle v,y\rangle|^2.
\]
\end{proof}

\subsection{Schur's theorem for quaternions}\label{subsec:Schurquat}
Recall $\mathbb{S}^n(\mathbb{F})\subset \mathbb{F}^{n\times n}$ is the space of $A$ satisfying $A^*=A$. We call such matrices self-adjoint. Assume that $A=(a_{ij})\in \mathbb{S}^n(\mathbb{H})$. We associate with $A$ the quaternion form $Q(x):=x^* Ax=\sum_{i=1}^n\sum_{j=1}^n\bar x_i a_{ij} x_j$ for $x\in \mathbb{H}^n$.  As $(x^* A x)^*=x^* Ax$ it follows that $Q(x)$ is always a real number.  The matrix $A$ is called positive semidefinite if $Q(x)\ge 0$ for all $x$.  
We denote by $\mathbb{S}^n_+(\mathbb{F})$ the cone of positive semidefinite self-adjoint matrices over $\mathbb{F}$. It is easy to check that $\mathbb{S}^n_+(\mathbb{H}) \cap \mathbb{R}^{n\times n} = \mathbb{S}^n_+(\mathbb{R})$. If $Q(x)> 0$ for all $x\neq 0$, $\langle x,y \rangle = x^* A y$ defines an inner product in $\mathbb{H}^n$.

Denote by $\mathbb{U}^n(\mathbb{F})\subset \mathbb{F}^{n\times n}$ the group of unitary matrices $U^* U=U U^*=I$.  The spectral theorem of $A\in  \mathbb{S}^n(\mathbb{F})$ claims that there exists a unitary $U$ and a real diagonal $D$ such that $A=U D U^*$ \cite{Farenick}.  The columns of $U$ are the eigenvectors of $A$ with real left eigenvalues, which are the corresponding diagonal entries of $D$.  Thus $A\in \mathbb{S}^n_+(\mathbb{F})$ if and only if all the left real eigenvalues of $A$ are nonnegative.  In that case $A$ has a unique square root $A^{1/2}=UD^{1/2}U^*\in \mathbb{S}^n_+(\mathbb{F})$.  Hence $A=\langle x_i,x_j\rangle$, where $x_1,\dots,x_n$ are the columns of $A^{1/2}$.  In particular, $\|x_1\|=\dots=\|x_n\|=1$ if and only if the diagonal entires of $A$ are $1$. To get the expression $A=\langle x_i,x_j\rangle$, we can also use the Cholesky decomposition. The usual algorithm for Cholesky decomposition works for quaternions.

\begin{lemma}\label{sapdquatcon}  Assume that $M\in\mathbb{H}^{n\times n}$ has representation $\hat C(M)\in \mathbb{C}^{(2n)\times (2n)}$ given by \eqref{hatCMrep}.
Then
\begin{enumerate}
\item $M\in \mathbb{S}^n(\mathbb{H})$ if and only if $Z\in\mathbb{S}^n (\mathbb{C})$ and $W$ is skew symmetric: $W^\tp=-W$.
\item $M\in \mathbb{S}^n(\mathbb{H})$ if and only if $\bar M\in \mathbb{S}^n(\mathbb{H})$.  Furthermore $M\in \mathbb{S}^n_+(\mathbb{H})$ if and only if $\bar M\in \mathbb{S}^n_+(\mathbb{H})$. 
\item $M\in \mathbb{S}_+(\mathbb{H})$ if and only if $\hat C(M)\in \mathbb{S}^{2n}_+(\mathbb{C})$.
\end{enumerate}
\end{lemma}
\begin{proof}
\emph{(1)}  Assume that $M=Z+W\mathrm{j}$, where $Z,W\in \mathbb{C}^{n\times n}$.  Then $M^*=Z^*-W^\tp \mathrm{j}$.  Thus $M^*=M$ if and only if $Z^*=Z$ and $-W^\tp=W$.  This is equivalent to the statement that $\hat C(M)\in \mathbb{S}^{2n}(\mathbb{C})$.

\emph{(2)} As $\bar M=\bar Z -W\mathrm{j}$ we deduce that $M$ is self-adjoint if and only if $\bar M$ is self-adjoint.  Suppose that $M\in \mathbb{S}^n_+(\mathbb{H})$.  Then $M=UDU^*$ where $U$ is unitary and $D$ is a real diagonal with nonnegative diagonal entires.  Then $\bar M=\overline{UDU^*}= \overline{U^*} D \bar U=U^\tp D \bar U$.  As $\bar U$ is unitary we deduce that $\bar M\in \mathbb{S}^n_+(\mathbb{H})$.  Similarly if $\bar M$ positive semidefinite then $M$ is positive semidefinite.

\emph{(3)}  Assume that $M$ is self-adjoint.   Then $\hat C(M)$ is positive semidefinite if and only 
$x^*Z x +y^*\bar Z y +2\Re x^*Wy\ge 0$ for 
$x, y\in \mathbb{C}^n$. 
Replace  $x, y$ with $x, -y$ we deduce that 
the above Hermitian form is nonnegative  if and only if the form $x^*Z x +y^*\bar Z y -2\Re x^*Wy$ is nonnegative. 

Assume that $M=Z+W\mathrm{j}\in \mathbb{S}^n(\mathbb{H})$.  Let $x=z+w\mathrm{j}\in \mathbb{H}^n$, where $z,w\in \mathbb{C}^n$.  A straightforward calculation shows:
\begin{eqnarray*}
x^* M x=(z^*-w^\tp \mathrm{j})(Z+W\mathrm{j})(z+w\mathrm{j})=z^* Z z+w^\tp \bar Z \overline{w}-z^*W\overline{w} +w^\tp \bar W z
\end{eqnarray*}
Clearly $(w^\tp \bar W z)^*=z^*W^\tp \overline{w}=-z^*W \overline{w}$.
 As $x^* M x$ is a real number it follows that 
$x^* M x=z^* Z z+w^\tp \bar Z \overline{w}-2\Re z^*W\overline{w}$.  Set $y=\overline{w}$ to deduce the claim.
\end{proof}

For $A=(a_{ij}), B=(b_{ij})\in \mathbb{F}^{m\times n}$ denote by $A\circ B=(a_{ij}b_{ij})$ the Schur product of two matrices.  Assume that $\mathbb{F}\in \{\mathbb{R},\mathbb{C}\}$.  Then the Schur product of two self-adjoint matrices is self-adjoint.
Furthermore, Schur's theorem claims that the Schur product of two positive semidefinite matrices is positive semidefinite.  Assume that $\mathbb{F}=\mathbb{H}$.
Since $\mathbb{H}$ is not commutative the product of two quaternion self-adjoint usually is not self-adjoint.  There are two simple exception:  Assume that $A\in \mathbb{S}^n(\mathbb{H})$.  If either $B\in \mathbb{S}^n(\mathbb{R})$ of $B=\bar A$
then $A\circ B$ is self-adjoint.  In these two cases Schur's theorem hold:

\begin{lemma}[The Schur product theorem for quaternions]
For symmetric positive semidefinite real matrix M and self-adjoint positive semidefinite quaternion matrix N, their Hadamard product, defined by $(M\circ N)_{ij} := M_{ij}N_{ij}$, is self-adjoint positive semidefinite. The matrix $L_{ij} := N_{ij}N_{ij}^* = \|N_{ij}\|^2$ is also self-adjoint positive semidefinite.
\end{lemma}
\begin{proof}
$M$ can be written as $M_{ij} = \langle a_i, a_j \rangle$ where $a_i \in \mathbb{R}^n$. And $N$ can be written as $N_{ij} = \langle x_i, x_j \rangle$ where $x_i \in \mathbb{H}^n$. By Lemma~\ref{teninprodform},
\[
\langle a_i\otimes x_k, a_j\otimes x_l\rangle =  \langle a_i,a_j\rangle \langle x_k,x_l\rangle.
\]
So both the Kronecker product and Schur product of $M$ and $N$ are semidefinite positive.

For the second claim, again by Lemma~\ref{teninprodform}, $L_{ij} = \|\langle x_i, x_j\rangle \|^2= \mathfrak{R}(\langle \bar{x}_i\otimes x_i, \bar{x}_j\otimes x_j \rangle)$. So $L$ is the real part of a positive semidefinite matrix and it is also positive semidefinite.
\end{proof}

\begin{lemma}\label{Gramexist}
Let $\mathcal{H}$ be a Hilbert space over quaternions. Assume that $x_1,\dots,x_n$ be $n$ unit vectors in $\mathcal{H}$.  Then for each $m\in\mathbb{N}$, there exists $m$ unit vectors $x_{1,m}\dots,x_{n,m}\in\mathbb{H}^n$ such that
$\langle x_{i,m}, x_{j,m}\rangle = \langle x_i, x_{j}\rangle |\langle x_i, x_{j}\rangle|^{2m}$.
\end{lemma}
\begin{proof} Let $A_0=(\langle  x_i,x_j\rangle)$.   Then $A_0\in \mathbb{S}^n_+$ is a correlation matrix.  Denote $A_m=A_{m-1}\circ (A_0\circ \bar A_0)$ for $m\in \mathbb{N}$.  Use induction and Lemma \ref{Schurquatlem} to deduce that $A_m$ is a correlation matrix. Then $A_m=(\langle  x_{i,m},x_{j,m}\rangle)$.  Hence $\langle x_{i,m}, x_{j,m}\rangle = \langle x_i, x_{j}\rangle |\langle x_i, x_{j}\rangle|^{2m}$.
\end{proof}

\subsection{The kernel trick}\label{kerneltrick}
We now state the kernel trick for quaternion Hilbert space \cite[\S3.7]{Vershynin}.  This technique that was used in \cite{FLZ}  to deduce in a unified way the Krivine-Haagerup upper bound on the  Grothendieck constant $K_G^{\mathbb{F}}$ for the fields of real or complex numbers $\mathbb{F}$.
\begin{lemma}\label{kerneltrick}
Let $\mathcal{H}$ be a Hilbert space over quaternions.  Assume that $x_1,\dots,x_n,y_1,\dots,y_n$ are $2n$ unit vectors in $\mathcal{H}$.  Suppose that $g(z)$ is an analytic function in the unit complex disk with Taylor series satisfying the following conditions.
\begin{eqnarray*}
g(z)=\sum_{i=0}^\infty a_m z^{2i}, \quad a_i\in\mathbb{R},\quad \sum_{i=0}^\infty |a_m|=1
\end{eqnarray*}
Then there exists $2n$ unit vectors $u_1,\dots,u_n,v_1,\dots,v_n$ in a Hilbert  space $\mathcal{H}'$ such that
\begin{eqnarray}\label{kernid}
\langle x_i, y_j\rangle g(|\langle x_i, y_j|)=
\langle u_{i}, v_{j}\rangle, \quad i,j=1,\dots,n.
\end{eqnarray}
\end{lemma}
\begin{proof}
Let $z_i=x_i, z_{n+i}=y_i$ for $i\in[n]$.  Lemma \ref{Gramexist} yields the existence of unit vectors  $w_{i,m}\in \mathbb{H}^{2n}$ such that $\langle w_{i,m}, w_{j,m}\rangle = \langle z_{i}, z_{j}\rangle |\langle z_{i}, z_{j}\rangle|^{2m}$ for each $m\in \mathbb{N}$ and $i,j\in[2n]$.

Let $\mathcal{H}'=\mathcal{H}\oplus (\oplus_{m=1}^\infty \mathbb{H}^{2n})$ with the corresponding induced inner product.  Define
\begin{eqnarray*}
&&u_i=x_i\sqrt{|a_0|}\oplus(\oplus_{m=1}^\infty w_{i,m}\sqrt{|a_m|}),\\
&&v_i=y_i\sign{a_0}\sqrt{|a_0|}\oplus(\oplus_{m=1}^\infty w_{n+i,m}\sign{a_m} \sqrt{|a_m|})
\end{eqnarray*}
for $i\in[n]$.  Then $u_i,v_i$ are unit vectors and \eqref{kernid} holds.
\end{proof}
\subsection{Existence of the Grothendieck constant for quaternions}\label{sec:existGK}
In this paper we view $\mathbb{H}^{m\times n}$ as a left vector space over $\mathbb{H}$.
We introduce two norms on $\mathbb{H}^{m\times n}$:
\begin{eqnarray}\label{defininftyquatnrm}
\|M\|_{\infty,1, \mathbb{H}}=\max\{|\sum_{i,j}^{m,n} M_{ij}\bar\varepsilon_i\delta_j|, \; \varepsilon_i,\delta_j\in\mathbb{H}, |\varepsilon_i|=|\delta_j|=1,i\in[m],j\in[n]\},\\
\label{defGrotquatnrm}
\|M\|_{G,\mathbb{H}}=\max\{|\sum_{i,j}^{m,n} M_{ij}\langle x_i, y_j\rangle|, x_i,y_i \in \mathcal{H}, \|x_i\|= \|y_j\|=1,i\in[m],j\in[n]\}.
\end{eqnarray}
Here $\mathcal{H}$ is a right Hilbert space over quaternions.
If we choose $\mathcal{H}$ to be one dimensional then the maximum in the Grothendieck norm is the maximum of $(\infty,1)$ norm.  Hence we have the inequality $\|M\|_{\infty,1}\le \|M\|_{G,\mathbb{H}}$.  Thus the problem if we have the reverse inequality independent on the dimensions $m,n$:  $K_{G}^{\mathbb{H}}\|M\|_{\infty,1}\ge \|M\|_{G,\mathbb{H}}$?
By multiplying each $\delta_j$ and $y_j$ by a fixed $a\in\mathbb{H}, |a|=1$  from the right, for $j\in[n]$, we can replace absolute values in the definitions of the norms $\|M\|_{\infty,1}, |M\|_{G,\mathbb{H}}$ by the real part
\begin{eqnarray}\label{quatinfty1norm}
\|M\|_{\infty,1,\mathbb{H}}=\max\{\Re(\sum_{i,j}^{m,n} M_{ij}\bar\varepsilon_i\delta_j), \; \varepsilon_i,\delta_j\in\mathbb{H}, |\varepsilon_i|=|\delta_j|=1,i\in[m],j\in[n]\},\\
\label{quatgrotnorm}
\|M\|_{G,\mathbb{H}}=\max\{\Re(\sum_{i,j}^{m,n} M_{ij}\langle x_i, y_j\rangle), x_i,y_j \in \mathcal{H}, \|x_i\|= \|y_j\|,i\in[m],j\in[n]\}.
\end{eqnarray}
Observe next that the maximum in the characterizations above we can replace the equalities 
$|\varepsilon_i|=|\delta_j|=\|x_i\|=\|y_j\|=1$ by the inequalities $|\varepsilon_i|,|\delta_j|,\|x_i\|,\|y_j\|\le 1$ \cite{FL20}.

Next we now are going to replace the maximum in the above characterization for $\|M\|_{\infty,1,\mathbb{H}}$ with quaternions with matrices of sizes $(2m)\times  (2n)$ with complex entries.  We start with the following lemma which follows by straightforward calculation:
\begin{lemma}\label{prod3quat}  Assume that the quaternions $\alpha, \varepsilon, \delta$ have the following matrix representations: 
\[\alpha=\left(\begin{array}{cc}a& b\\-\bar b&\bar a\end{array}\right),\varepsilon=\left(\begin{array}{cc}z&w\\-\bar w&\bar z\end{array}\right), \delta=\left(\begin{array}{cc}u&v\\-\bar v&\bar u\end{array}\right)\in \mathbb{C}^{2\times 2}.\]
Then
\[\Re(\alpha\bar\varepsilon\delta)=\Re((\bar z,\bar w)\left(\begin{array}{cc}a&-\bar b\\ b&\bar a\end{array}\right)(u,v)^{\tp})=\Re((\bar z,\bar w) A(\alpha)^{\tp}(u,v)^{\tp}).\]
\end{lemma}
\begin{lemma}\label{compverinfty1quat} Let 
\[M=(M_{ij})\in \mathbb{H}^{m\times n}, \varepsilon_i=(z_i,w_i), \delta_{j}=(u_i,v_i)\in\mathbb{H}, z_i,w_i,u_j,v_j\in\mathbb{C}, \]
where $|z_i|^2+|w_i|^2=|u_i|^2+|v_i|^2=1, i\in[m],j\in[n]$.
To each quaternion $M_{ij}=(M_{ij,1},M_{ij,2})$ associate the matrix $C(M_{ij})=\left(\begin{array}{cc} M_{ij,1}&M_{ij,2}\\ -\bar M_{ij,2}&\bar M_{ij,1}\end{array}\right)$.  Let  $\tilde M=(C(M_{ij})^\tp)\in \mathbb{C}^{(2m)\times (2n)}$ and
\begin{eqnarray*}
e=(\bar z_1,\bar w_1,\bar z_2, \bar w_2,\dots,\bar z_m,\bar w_m)^\tp\in\mathbb{C}^{2m}, \quad
d=(u_1,v_1, u_2,v_2,\dots,u_n, v_n)^\tp\in\mathbb{C}^{2n}.
\end{eqnarray*}
Then
\begin{eqnarray*}
Re(\sum_{i,j}^{m,n} M_{ij}\bar\varepsilon_i\delta_j) =\Re(e^\tp\tilde Md) ,
\textrm{ where } \varepsilon_i,\delta_j\in\mathbb{H}, |\varepsilon_i|=|\delta_j|=|z_i|^2+|w_i^2|=|u_j|^2 + |v_j|^2=1, 
\end{eqnarray*}
and
\begin{equation}\label{equivdefint1quat}
\|M\|_{1,\infty,\mathbb{H}}=\max\{\Re(e^\tp\tilde Md), e\in\mathbb{C}^{2m},d\in\mathbb{C}^2n, |z_i|^2+|w_i^2|=|u_j|^2 + |v_j|^2=1,
\end{equation}
for $i\in[m],j\in[n]$.  In particular
\begin{equation}\label{ineqdifnorms}
\|M\|_{1,\infty,\mathbb{H}}\le \|\tilde M\|_{1,\infty,\mathbb{C}}\le 2\|M\|_{1,\infty,\mathbb{H}}.
\end{equation}
 \end{lemma}
\begin{proof} Use Lemma \ref{compverinfty1quat} to obtain the first equality of the lemma.
Clearly, the set $|z|^2+|w|^2=1$ is a subset of $|z|\le 1, |w|\le 1$.  Hence the characterizations \eqref{equivdefint1quat} and the complex norm $\|\tilde M\|_{\infty,1,\mathbb{C}}$ yield the first inequality in \eqref{ineqdifnorms}.
In the equality $|z|^2+|w|^2=1$ choose a special case $|z|=|w|=\frac{1}{\sqrt{2}}$.  Then $|\sqrt{2}z|=|\sqrt{2}w|=1$.   Hence the maximal characterization for $2\tilde M$ corresponding to $\|2M\|_{\infty,1,\mathbb{H}}$ is not less than $\|\tilde M\|_{\infty,1,\mathbb{C}}$.
\end{proof}

We now recast $\|M\|_{G,\mathbb{H}}$ in termis of the above matrix $\tilde M$.
We first write down $x,y\in\mathbb{H}^l$ as $z+w\mathrm{j}, u+v\mathrm{j}$, where $z,w,u,v\in\mathbb{C}^l$.  
We next observe that 
\begin{equation}\label{inprodquatfor}
\langle x,y\rangle=\langle z+w\mathrm{j},u+v\mathrm{j}\rangle=(\langle z,u\rangle+
\langle v,w\rangle)+
(\langle z,v\rangle -\langle u,w\rangle)\mathrm{j}.
\end{equation}
Introduce the following four vectors:
\[f_1=(z^\tp,\overline{w}^\tp)^\tp,\,f_2=(w^\tp,-\overline{z}^\tp)^\tp,\,g_1=(u^\tp,\overline{v}^\tp)^\tp,\,g_2=(v^\tp,-\overline{u}^\tp)^\tp\in\mathbb{C}^{2l}. \]
Then
\[\langle x,y\rangle=\langle f_1,g_1\rangle +\langle f_1,g_2\rangle\mathrm{j}=\overline{\langle f_2,g_2\rangle}-\overline{\langle f_2,g_1\rangle}\mathrm{j}.\]
Note that $\langle f,g\rangle=f^*g$ is the standard inner product on $\mathbb{C}^{2l}$.  Furthermore
\[\|f_1\|=\|f_2\|=\|x\|, \quad \|g_1\|=\|g_2\|=\|y\|,\]
The next lemma relates $\Re(\alpha\langle x,y\rangle)$ to the corresponding $\Re(\sum_{i,j=1}^2 (A(\alpha)^{\top})_{ij}\langle f_i,g_j\rangle)$:
\begin{lemma}\label{grotquatcompl}  Let $\alpha=a+b\mathrm{j}\in \mathbb{H}$ and $x,y\in\mathbb{H}^l$.  Then 
\[\Re(\alpha\langle x,y\rangle)=\frac{1}{2}\Re(\sum_{p,q=1}^2 (A(\alpha)^\tp)_{pq}\langle f_p,g_q\rangle).\]
\end{lemma}
\begin{proof}The above formulas yield
\begin{eqnarray*}
\Re(\alpha\langle x,y\rangle)&&=\Re(a(\langle z,u\rangle+
\langle v,w\rangle)-b(\overline{\langle z,v\rangle -\langle u,w\rangle})=
\Re(a(\langle u,z\rangle+
\langle w,v\rangle)-b(\langle v,z\rangle-\langle w,u\rangle )\\
&&=\Re(a \langle f_1,g_1\rangle +b\langle f_2,g_1\rangle)=\Re(\bar a \langle f_2,g_2\rangle -\bar b\langle f_1,g_2\rangle)
\end{eqnarray*}
As $A(\alpha)=\left(\begin{array}{cc}a& b\\-\bar b&\bar a\end{array}\right)$ we get
\[\Re(\sum_{p,q=1}^2 (A(\alpha)^\tp)_{pq}\langle f_p,g_q\rangle)=\Re(a\langle f_1,g_1\rangle-\bar b \langle f_1,g_2\rangle+b\langle f_2,g_1\rangle+\bar a\langle f_2,g_2\rangle).\]
Compare the two expressions to deduce the lemma.
\end{proof}

The following result is an improvement of the Grothendieck's result \cite{Grothendieck} that $K_G^{\mathbb{C}}\le 2K_G^{\mathbb{R}}$,  and an analog of \cite[(63)]{FL20}:
\begin{theorem}\label{quatlescomp}  The Grothendieck constant for quaternions is not more then the Grothendieck constant for complex numbers: $K_G^{\mathbb{H}}\le K_G^{\mathbb{C}}$.
\end{theorem} 
\begin{proof}   Let $M=(M_{ij})\in \mathbb{H}^{m\times n}$.  Let $\tilde M\in\mathbb{C}^{(2m)\times (2n)}$ be defined as in the proof of Lemma \ref{compverinfty1quat}.  Lemma \ref{compverinfty1quat} claimes that $\|\tilde M\|_{\infty,1,\mathbb{C}}\le 2 \|\tilde M\|_{\infty,1,\mathbb{H}}$.

Assume that the entries of $\tilde M$ are $\hat M_{ij}\in\mathbb{C}$ where $i\in[2m],j\in[2n]$.
Let $x_1,\dots,x_m,y_1,\dots,y_n
\in\mathcal{H}$, be vectors of norm one, where $\mathcal{H}$ is an inner product space over quaternions.  As one has Gram-Schmidt process in $\mathcal{H}$ we can assume that  $x_1,\dots,x_m,y_1,\dots,y_n\in \mathbb{H}^{n+m}$. 
Consider the Grothendieck norm $\|M\|_{G,\mathbb{H}}$.  For each $x_i, y_j$ define $f_{2i-1}, f_{2i}, g_{2j-1}, f_{2j}$ as before the proof of Lemma \ref{grotquatcompl}.  The proof of Lemma \ref{grotquatcompl} that $\|M\|_{G,\mathbb{H}}$ is maximum on all $f_1,\dots,f_{2m},g_1,\dots,g_{2n}\in \mathbb{C}^{2(n+m)}$ of the expression
\[\frac{1}{2} \Re(\sum_{i,j=1}^{m,n} \sum_{p,q=1}^2 (C(M_{i,j})^\tp)_{pq}\langle f_{2(i-1)+p},g_{2(j-1)+q}\rangle).\]
Clearly, the above maximum is not more then the maximum for $\frac{1}{2}\|\tilde M\|_{G,\mathbb{C}}$, since the vectors $f_{2i-1}, f_{2i}, g_{2j-1}, f_{2j}$ are of a special form.  Hence 
\[\|M\|_{G,\mathbb{H}}\le \frac{1}{2}\|\tilde M\|_{G,\mathbb{C}}\le \frac{K_G^{\mathbb{C}}}{2}\|\tilde M\|_{\infty,1,\mathbb{C}}\le K_{G}^{\mathbb{C}}\|M\|_{\infty,1,\mathbb{H}}.\]
(The right hand side inequality follows from \eqref{ineqdifnorms}.)
\end{proof}
\subsection{Semidefinite programming for computing quaternion Grothendieck norm} \label{subsec: quatSDP}
In this subsection we state the computation of $\|M\|_{G,\mathbb{H}}$  as an SDP problem on positive semidefinite Hermitian matrices.  We use the characterization \eqref{quatgrotnorm}.  Let $z_i=x_i, z_{n+j}=y_j$ for $i\in[m], j\in [n]$.  Denote by $\mathbb{G}^p(\mathbb{H})\subset \mathbb{S}^p_+(\mathbb{H})$ the convex set of quaternion correlations matrices, i.e., all positive semidefinite quaternionic matrices whose diagonal entries are $1$.   Assume that   $G=G(z_1,\dots,z_{m+n})\in \mathbb{G}^{m+n}(\mathbb{H})$.  Let $H=\hat C(G)$.  Then $H\in \mathbb{S}^{2(m+n)}(\mathbb{C})$ is a complex correlation matrix of the form \eqref{hatCMrep}, where $Z,W\in \mathbb{C}^{m+n}$ and $W^\tp=-W$.  Let us denote 
this real subspace of complex correlation matrices by $\mathcal{C}^{2(m+n)}$.  Define as in \cite{FL20}
\begin{eqnarray*}
A(M) = \begin{bmatrix}  0 & M \\ M^* & 0 \end{bmatrix} \in \mathbb{S}^{m+n}(\mathbb{H})
\end{eqnarray*}
\begin{lemma}\label{SDPGHchar}  Assume that $M\in \mathbb{H}^{m\times n}$.  Then
\begin{eqnarray*}
\|M\|_{G,\mathbb{H}}=\frac{1}{2}\max\{\Re\tr A(M)G,G\in \mathbb{G}^{m+n}(\mathbb{H})\}= 
\frac{1}{4}\max\{\tr \hat C(A(M))H, H\in \mathcal{C}^{2(m+n)}\}.
\end{eqnarray*}
\end{lemma}
\begin{proof}  Let $G=G(z_1,\dots,z_{m+n})$ be defined as above.
Then 
\begin{eqnarray*}
&&\tr A(M) \bar G= \sum_{i=1}^m \sum_{j=1}^n M_{ij}\overline{\langle y_j,x_i\rangle} +\overline{M_{ji}}\overline{\langle x_i,y_j\rangle}=\sum_{i=1}^m \sum_{j=1}^n M_{ij}\langle x_i,y_j\rangle +\overline{M_{ji}}\overline{\langle x_i,y_j\rangle},\\
&&\Re\tr A(M) \bar G= \sum_{i=1}^m \sum_{j=1}^n \Re M_{ij}\langle x_i,y_j\rangle +\Re\overline{M_{ij}}\overline{\langle x_i,y_j\rangle}=2 \Re\sum_{i=1}^m \sum_{j=1}^n M_{ij}\langle x_i,y_j\rangle
\end{eqnarray*}
(To deduce the last equality we used \eqref{tracecomrel}.)  
Lemma  \ref{sapdquatcon} yields that $\bar G\in\mathbb{S}^{m+n}_+$ iff and only if $\bar G\in\mathbb{S}^{m+n}_+$.  As the diagonal entries of $\bar G$ are $1$ we deduce that $\bar G\in \mathbb{G}^{m+n}(\mathbb{H})$ if and only if $G\in \mathbb{G}^{m+n}(\mathbb{H})$.
Use \eqref{quatgrotnorm} to deduce the first part of the characterization of $\|M\|_{G,\mathbb{H}}$.
Since $A(M),G\in \mathbb{S}^{m+n}(\mathbb{H})$ we have 
\begin{eqnarray*}
&&M=M_1+M_2\mathrm{j}, \;A(M)=A_1+A_2\mathrm{j}=\begin{bmatrix} 0 &M_1\\ M_1^*&0\end{bmatrix} +\begin{bmatrix} 0 &M_2\\ -M_2^\tp&0\end{bmatrix}\mathrm{j}, \;G=G_1+G_2\mathrm{j},\\ 
&&M_1,M_2\in \mathbb{C}^{m\times n}, \quad G_1\in \mathbb{S}^{m+n}(\mathbb{C}),\quad
G_2\in \mathbb{C}^{(m+n)\times (m+n)},\quad  G_2^\tp =-G_2. 
\end{eqnarray*}
Hence
\begin{eqnarray*}
\bar G=\bar G_1 -G_2\mathrm{j}, \quad \Re\tr A(M)\bar G=\Re(A_1\bar G_1+A_2 \bar G_2).
\end{eqnarray*}
Observe next
\begin{eqnarray*}
\hat C(A(M))=\begin{bmatrix}A_1 &A_2\\-\bar A_2& \bar A_1\end{bmatrix}, \quad \hat C(\bar G)=\begin{bmatrix}\bar G_1 &-G_2\\\bar G_2& G_1\end{bmatrix},\\ 
\tr \hat C(A(M))\hat C(\bar G))=\tr (A_1\bar G_1 +A_2\bar G_2 +\bar A_2 G_2 +\bar A_1 G_1)=2 \Re (A_1\bar G_1+A_2 \bar G_2).
\end{eqnarray*}
As $\bar G\in\mathbb{G}^{m+n}(\mathbb{H})$ we deduce that $\hat C(\bar G)\in \mathcal{C}^{2(m+n)}$.  Vice versa, if $H\in \mathcal{C}^{2(m+n)}$ then $H=\widehat {C(\bar G)}$ for some $\bar G\in\mathbb{G}^{m+n}(\mathbb{H})$.
  Hence $G$ is a quaternion correlation matrix.  This proves the lemma.
\end{proof}

\subsection{The sign function for quaternions}
We define our sign function over $ \mathbb{F}$ 
\begin{equation}\label{eq:sign}
\sign z= \begin{cases}
z/\lvert z \rvert & z\ne 0,\\
0 & z = 0.
\end{cases}
\end{equation} 
Denote by $\mathrm{S}^3=\{a\in\mathbb{H}, \; |a|=1\}$, the $3$-dimensional sphere in $\mathbb{R}^4$.  Note that multiplication by $\phi_a(b)= ab$, and $\psi_a(b)= ba$ are orientation preserving orthogonal transformation on $\mathbb{H}$ for a fixed $a\in\mathrm{S}^3$ and $b\in\mathbb{H}$.  In particular $\phi_a(\mathrm{S}^3)=\psi_a(\mathrm{S}^3)=\mathrm{S}^3$.

On $\mathrm{S}^3$ let $d\sigma$ be the Haar measure on $\mathrm{S}^3$, which is invariant under the action of $\phi_a, \psi_b$.  We now give the following generalization of Haagerup formula \cite{Haagerup} for $\sign(z)$ for quaternions:
\begin{lemma}\label{sgnforquat}
Let $z\in\mathbb{H}$.  Then
\begin{equation}\label{sgnforquat1}
\sign(z)=\frac{3}{8\pi}\int_{w\in\mathrm{S^3}} \sign(\Re (\bar wz))wd\sigma(w).
\end{equation}
\end{lemma}
\begin{proof} Clearly for $z=0$ \eqref{sgnforquat1} trivially holds.  We next assume that 
$z=1$.  Hence the left hand side of \eqref{sgnforquat1} is $1$.  Let $w=w_0+w_1\mathrm{i}+w_2\mathrm{j}+w_3\mathrm{k}\in\mathrm{S^3}$.
Then $\Re(w^*)=w_0$ and $\sign(\Re(\bar w))=\sign(w_0)$.  We first observe that
\[\int_{w\in\mathrm{S^3}}\sign(w_0) w_jd\sigma(w)=0 \textrm{ for } j\in[3].\]
This follows from the observation that the transformation $w\mapsto \bar w$ is Haar measure preserving on $\mathrm{S^3}$.  Hence  
\[\int_{w\in\mathrm{S^3}}\sign(w_0) w_jd\sigma(w)=- \int_{w\in\mathrm{S^3}}\sign(w_0) w_jd\sigma(w), \quad j\in[3].\] 
Observe next that  $\sign(\Re(\bar w))w_0=\sign (w_0) w_0=|w_0|$.
Thus we need to show that 
$\int_{w\in\mathrm{S^3}}|w_0|d\sigma(w)=\frac{8\pi}{3}$.
We now introduce the spherical coordinates on $\mathbb{R}^4$ as follows:
\begin{align*}
&w_0=\cos (\phi_0), w_1=\sin (\phi_0)\cos (\phi_1), w_2=\sin (\phi_0)\sin(\phi_1) \cos(\phi_2),\\
&w_3=\sin(\phi_0)\sin(\phi_1)\sin(\phi_2), d\sigma=\sin^2(\phi_0)\sin(\phi_1)d\phi_0 d\phi_1d\phi_2,\\
&\phi_0,\phi_1\in[0,\pi],\phi_2\in[0,2\pi].\label{spheric2}
\end{align*}
Hence 
\[\int_{w\in\mathrm{S^3}}|w_0|d\sigma(w)=\big(\int_{0}^{\pi}|\cos(\phi_0)|\sin^2(\phi_0)d\phi_0\big)\big(\int_{0}^{\pi}\sin(\phi_1)d\phi_1\big)\big(\int_0^{2\pi}d\phi_2
\big)=\frac{8\pi}{3}.\]
Hence \eqref{sgnforquat1} holds for $z=1$.  For a general $z\ne 0$ we recall that $\sign(z)=\sign(tz)$ for any $t>0$.  Hence it is ehough to show \eqref{sgnforquat1} for $z\in \mathrm{S^3}$.  That is we need to show the equality
\[z=\frac{3}{8\pi}\int_{w\in\mathrm{S^3}} \sign(\Re (\bar w z))wd\sigma(w), \quad z\in \mathrm{S^3}\]
By multiplying by $\bar z$ from the left it is enough to show that 
\[1=\frac{3}{8\pi}\int_{w\in\mathrm{S^3}} \sign(\Re (\bar w z))\bar z wd\sigma(w).\]
Now introduce a new variable $u=\phi_{\bar z}(w)=\bar z w$ on $\mathrm{S^3}$ .  Note that $\bar u=\bar w z$.  Since the Haar measure on $\mathrm{S^3}$ is invariant under $\phi_{\bar z}$ we get that $d\sigma (w)=d\sigma(u)$.  Hence the above equality is equivalent to \eqref{sgnforquat1} for $z=1$, which was proved.
\end{proof}
\subsection{Quaternion Gaussian}\label{subsec:quatgaus}
Recall the distribution of $G_n^\mathbb{R}(z)$ and $G_n^\mathbb{C}(z)$
given \cite{Haagerup}. The Gassian quaternions $\mathbb{H}^n$ has the distribution
\[G_n^\mathbb{H}(z) =\big(\frac{\pi}{2}\big)^{-2n} \exp(-2\|z\|_2^2).\]
The variance chosen here is totally arbitrary.

\begin{theorem}\label{uvfuvfor}  Assume that $u,v\in\mathcal{H}$ are of norm one, where $\mathcal {H}$ is a right vector space over $\mathbb{H}$.  $m(z)$ is the Lebesgue measure in $\mathbb{H}^n$, then
\begin{align*}\label{mainforuvfuv}
\int_{\mathbb{H}^n}\sign\langle u,z\rangle\sign\langle z,v \rangle G_n^\mathbb{H}(z)dm(z) &=\langle u,v\rangle f_{\mathbb{H}}(|\langle u,v\rangle|)\\
&= \langle u,v\rangle \frac{3}{2}\int_0^{\frac{\pi}{2}}\frac{\cos^4 t}{\sqrt{1-|\langle u,v\rangle|^2 \sin ^2 t}}dt.
\end{align*}
If $\langle u,v\rangle$ is real, then
\begin{align*}
\int_{\mathbb{H}^n}\sign\langle z,u\rangle\sign\langle v,z \rangle G_n^\mathbb{H}(z)dm(z) &=\langle u,v\rangle f_{\mathbb{H}}(|\langle u,v\rangle|)\\
&= \langle u,v\rangle \frac{3}{2}\int_0^{\frac{\pi}{2}}\frac{\cos^4 t}{\sqrt{1-|\langle u,v\rangle|^2 \sin ^2 t}}dt.
\end{align*}
\end{theorem}
\begin{proof}  
We will prove the first formula, the proof of second formula is the same. First we claim it is enough to prove this formula for $n=2$.
Indeed, Since span $(u,v)$ is at most two dimensional, by performing Gram-Schmidt orthogonalization process we can find an orthonormal basis in $\mathbb{F}^n$ such that $u,v\in$ span $(e_1,e_2)$. $\mathbb{F}^{n}$ can be decomposed as span $(e_1,e_2)$ and its complement. By the correspond decomposition of the Gaussian variable, its integration on the complement dimension is simply 1.

Recall the famous Grothendieck’s Identity: for any fixed real vectors $u,v$ of norm 1, we have
\[
\int_{\mathbb{R}^n} \sign\langle x,u\rangle \sign\langle x,v\rangle G_n^\mathbb{R}(x) dx = \frac{2}{\pi} \arcsin \langle u,v\rangle.
\]

For $u,v,z \in \mathbb{H}^n$, donote their real vector form as $u,v,z \in \mathbb{R}^{4n}$. Then $\langle u,v\rangle = \Re \langle u,v\rangle$ and the Grothendieck’s Identity becomes:
\[
\int_{\mathbb{H}^n} \sign \Re \langle u,z\rangle \sign \Re \langle z,v\rangle G_n^\mathbb{H}(z) dm(z) = \frac{2}{\pi} \arcsin \Re \langle u,v\rangle.
\]

By Lemma~\ref{sgnforquat}, we have
\begin{align*}
&\int_{\mathbb{H}^n}\sign\langle u,z\rangle \sign \langle z,v \rangle G_n^\mathbb{H}(z)dm(z)\\
= &\frac{9}{64\pi^2}\int_{w_1,w_2 \in\mathrm{S^3}}\int_{z \in \mathbb{H}^n} \sign(\Re (\bar w_1\langle u,z\rangle))w_1 \sign(\Re (\bar w_2\langle z,v\rangle))w_2 d\sigma(w_1)d\sigma(w_2) G_n^\mathbb{H}(z) dm(z)\\
= &\frac{9}{32\pi^3}\int_{w_1,w_2 \in\mathrm{S^3}}\arcsin(\Re (\langle uw_1,v\bar w_2\rangle))w_1w_2 d\sigma(w_1)d\sigma(w_2)
\end{align*}

(1) Assume now $\langle u,v\rangle = a \in \mathbb{R}$, then $\Re \langle u w_1, v\bar w_2\rangle=a\Re(\bar w_1\bar w_2) = a\Re(w_1 w_2)$.
Thus we deduce 
\begin{align*}
\langle u,v\rangle f_{\mathbb{H}}(|\langle u,v\rangle|)&=\frac{9}{32\pi^3}\int_{w_1,w_2\in\textrm{S}^3} \arcsin(|\langle u,v\rangle|\Re (w_1 w_2))w_1 w_2d\sigma(w_1)d\sigma(w_2)\\
&=\frac{9}{32\pi^3}\int_{w_1,w_2\in\textrm{S}^3} \arcsin(|\langle u,v\rangle|\Re (w_1 (w_1^{-1}w_2))w_1 (w_1^{-1}w_2) d\sigma(w_1)d\sigma(w_2)\\
&=\frac{9}{16\pi}\int_{q\in\textrm{S}^3} \arcsin(|\langle u,v\rangle|\Re (q)) q  d\sigma(q)
\end{align*}
The second equality is due to the fact that $d\sigma(w_2)$ is a Haar measure. The third equality is due to the fact that the volume of $\textrm{S}^3$ is equal to  $2\pi^2$.

Observe that the left hand side of this equality is real.  Hence
\begin{equation}\notag
\langle u,v\rangle f_{\mathbb{H}}(|\langle u,v\rangle|)=\frac{9}{16\pi}\int_{q\in\textrm{S}^3} \arcsin(|\langle u,v\rangle|\Re (q))\Re(q)d\sigma(q).
\end{equation}
Now use the spherical coordinates as before to deduce that 
\begin{align*}
\langle u,v\rangle f_{\mathbb{H}}(|\langle u,v\rangle|)&=\frac{9}{4}\int_0^\pi \arcsin(|\langle u,v\rangle| \cos\phi_0)\cos\phi_0 \sin^2\phi_0d\phi_0\\
&= \frac{9}{2}\int_0^{\frac{\pi}{2}} \arcsin(|\langle u,v\rangle| \cos\phi_0)\cos\phi_0 \sin^2\phi_0d\phi_0\\
&=\frac{9}{2}\int_0^{\frac{\pi}{2}} \arcsin(|\langle u,v\rangle| \sin t)\sin t  \cos^2 t\,dt.
\end{align*}
Finally do the integration by part to get
\[
\langle u,v\rangle f_{\mathbb{H}}(|\langle u,v\rangle|)=\frac{3|\langle u,v\rangle|}{2}\int_0^{\frac{\pi}{2}}\frac{\cos^4 t}{\sqrt{1-|\langle u,v\rangle|^2 \sin ^2 t}}dt.
\]
Thus
\begin{equation}\label{fHform}
f_{\mathbb{H}}(|\langle u,v\rangle|)=
\frac{3}{2}\int_0^{\frac{\pi}{2}}\frac{\cos^4 t}{\sqrt{1-|\langle u,v\rangle|^2 \sin ^2 t}}dt.
\end{equation}

(2) If $\langle u,v\rangle$  is not real, then there is a norm 1 quaternion $c$ such that $\langle u c,v\rangle \in \mathbb{R}$.
\begin{align*}
&\int_{\mathbb{H}^n}\sign\langle u,z\rangle \sign \langle z,v \rangle G_n^\mathbb{H}(z)dm(z)\\
= &\int_{\mathbb{H}^n}c\sign\langle u c,z\rangle \sign \langle z,v \rangle G_n^\mathbb{H}(z)dm(z)\\
= &c \frac{3\langle cu,v\rangle}{2}\int_0^{\frac{\pi}{2}}\frac{\cos^4 t}{\sqrt{1-|\langle cu,v\rangle|^2 \sin ^2 t}}dt\\
= &\frac{3\langle u,v\rangle}{2}\int_0^{\frac{\pi}{2}}\frac{\cos^4 t}{\sqrt{1-|\langle u,v\rangle|^2 \sin ^2 t}}dt
\end{align*}
\end{proof}
\subsection{The function $p(x)$}
Define a function over the open \emph{quaternion} unit disk $D_{\mathbb{H}}=\{z\in\mathbb{H}, |z|< 1\}$ by
\begin{equation}\label{defP(z)}
P(z) \coloneqq \frac{3z}{2}\int_0^{\pi / 2}\frac {\cos^4t}{(1-|z|^2\sin^2t)^{1/2}}\,dt, \qquad z\in D_{\mathbb{H}} , 
\end{equation}
and the function $p(x)$ as the restriction of $P$ to $(-1,1) \subseteq \mathbb{R}$. Note that $p^\prime (x)>0$ on $(-1,1)$, $p(-1)=-1$ and $p(1)=1$. Hence $p:[-1,1]\to [-1,1]$ is a strictly increasing continuous bijection. Since $[-1,1]$ is compact, $p$ is a homeomorphism of $[-1,1]$ onto itself. By the Taylor expansion
\begin{gather*}
(1-x^2\sin^2t)^{-1 /2}=\sum_{k=0}^{\infty}\frac{(2k-1)!!}{(2k)!!}x^{2k}\sin^{2k}t, \qquad |x|\leq1,\; 0\leq t< \pi/ 2,
\shortintertext{and}
\int_0^{\pi / 2}\cos^4t\sin^{2k}t\,dt=\frac{3\pi}{2}\cdot\frac{(2k-1)!!}{(2k+4)!!}, 
\end{gather*}
we get
\begin{equation}\label{taylorforp}
p(x)=\sum_{k=0}^{\infty}\frac{9\pi}{16(k+1)(k+2)}\biggl[\frac{(2k-1)!!}{(2k)!!}\biggr]^2x^{2k+1}, \qquad x\in[-1,1].
\end{equation}
Let $p_{\ell}(x)=x\; \Hypergeometric{2}{1}{\frac{1}{2}, \frac{1}{2}}{\ell}{x^2}$ be the function introduced in \S\ref{sec:intro}.
A straightforward calculation shows that $p(x)=\frac{9\pi}{32}p_{3}(x)$.

This observation coincides with the formula \cite[(3.2)]{Bri2}.  Namely $\mathcal{E}_4(z)=p(x)$.  More general $\mathcal{E}_{2d}(z)=C_{2d}p_{d+1}(x)$, where $C_{2d}=(1/d)\bigl(\Gamma((2d+1)/2)/\Gamma(d)\bigr)^2$.  Thus, whenever the inverse function of $p_{\ell}(x)$ has first Taylor coefficient positive and all other nonpositive one can improve the value of the Grothendieck constants $K_{G, 2(\ell-1)}^{\mathbf{R}}$ as in \cite{FL20}.

Compare above $p(x)$ with the real Haagerup function 
\begin{equation}\label{hfuncfor}
h(x)=\sum_{k=0}^{\infty}\frac{\pi}{4(k+1)}\biggl[\frac{(2k-1)!!}{(2k)!!}\biggr]^2x^{2k+1}, \qquad x\in[-1,1].
\end{equation}
Observe that $h(x)=\frac{\pi}{4}p_{2}(x)$.
Note that
\[(x^3 p(x))'=\frac{3}{2}x^2h(x).\]

As $\biggl[\frac{(2k-1)!!}{(2k)!!}\biggr]<1$ It follows that 
\[\sum_{k=0}^{\infty}\frac{9\pi}{16(k+1)(k+2)}\biggl[\frac{(2k-1)!!}{(2k)!!}\biggr]<\sum_{k=0}^{\infty}\frac{9\pi}{16(k+1)(k+2)}=\frac{9\pi}{16}\sum_{k=0}^{\infty}(\frac{1}{k+1}-\frac{1}{k+2})=\frac{9\pi}{16}.\]
Therefore the power series for $p(x)$ (or $P(z)$) converge uniformly to a continuous function on the closed quaternion unit disk $\bar D_{\mathbb{H}}$.  Note that $p(z)$ is analytic in the open complex unit disk $D=\{z\in\mathbb{C}, |z|<1\}$.  Use the ratio test for the coefficients to see that the radius of convergence of the series for $p(z)$ is $r=1$.  Since the Taylor coefficients  of $p(z)$ are nonnegative, Pringsheim's theorem yield that $z=1$ is a singular point.  As $p(-z)=-p(z)$ it follows that $z=-1$ is also singular point.
As $p'(0) >0$ it follows that $p(z)$ has an inverse analytic function in some disc $D(r)=\{z\in\mathbb{C}, |z|<r\}$.  So 
\begin{equation}\label{invexpansionp}
p^{-1}(z)=\sum_{k=0}^{\infty}c_{2k+1}z^{2k+1},\qquad z\in \mathbb{C}, |z|<r, 0<r \le 1.
\end{equation}
The reason that $r\le 1$ is because $1$ is the singular point of $p(z)$. The coefficients $c_{2k+1}$ are given by the Lagrange inversion formula:
\begin{equation}\label{inversion_formula}
c_{2k+1}=\frac{1}{(2k+1)!}\lim_{t\to 0}\biggl[\frac{d^{2k}}{dt^{2k}}\left(\frac{t}{p(t)}\right)^{2k+1}\biggr].
\end{equation}
Since $p'(x)>0$ for $x\in(-1,1)$ the function $p^{-1}(z)$ is an analytic function in some simply connected domain containing $(-1,1)$.   

Assume that $p(z)=w$.  Then $z=p^{-1}(w)$.  As $p_3(z)=\frac{32}{9\pi}w$ it follows that $z=p_3^{-1}(\frac{32}{9\pi}w)=p^{-1}(w)$.  Hence the Taylor coefficients of $p^{-1}(w)$ and $p^{-1}_3(w)$ have the same signs.

\subsection{Haagerup's method}  In this subsection we will try to apply methods in \cite{Haagerup} to show that $c_{2k+1}<0$ for $k>0$.
We first show that as in \cite[Lemma 2.2]{Haagerup} that $p(z)$ can be extended to a continuous function  $p^{+}(z) $ in the closed upper half plane $\mathbb{C}^+=\{z\in\mathbb{C}, \Im (z)\ge 0\}$ which is analytic in the open upper half plane $\mathbb{C}^+_o=\{z\in\mathbb{C},\Im (z)>0\}$.  
Recall that 
\begin{equation}\label{defpxint}
p(x)=\frac{9}{2}\int _0^{\frac{\pi}{2}} \sin t \cos^2 t \arcsin (x \sin t) dt, \quad -1\le x \le 1.
\end{equation}
The analytic function $\sin z$ is a bijection of $[-\frac{\pi}{2},\frac{\pi}{2}]\times [0,\infty)$ onto the closed upper half plane.  Let $\arcsin^+ z: \mathbb{C}^+\to [-\frac{\pi}{2},\frac{\pi}{2}]\times [0,\infty)$. Note that 
\begin{align*}
\arcsin^+ x&=\arcsin x \textrm{ for } -1\le x\le 1,\\
\arcsin^{+} x&=\frac{\pi}{2} +\textbf{i}\,\arccosh x  \textrm{ for } x\ge 1,\\
\arcsin^{+} x&=-\frac{\pi}{2} +\textbf{i}\,\arccosh(-x)  \textrm{ for } x\le -1. 
\end{align*}
Furthermore $\arcsin ^{+} z$ is analytic in $\mathbb{C}_o^+$.  Hence we can define
\begin{equation}\label{defp+}
p^{+}(z)=\frac{9}{2}\int_0^{\frac{\pi}{2}} \sin t \cos^2 t \arcsin^{+}(z\sin t)dt.
\end{equation}
\begin{lemma}\label{Halem2.2}
The function $p(x)$ given by \eqref{defpxint} has an analytic extension to $\mathbb{C}^+_o$ given \eqref{defp+}.  Furthemore $p^+(z)$ is continuous on $\mathbb{C}^+$.  Its value for $x>1$ is given by the formulas:
\begin{gather*} 
\Re (p^{+}(x))=\frac{9}{2}(\int_{0}^{\sin t=\frac{1}{x}}\sin t \cos^2 t \arcsin(x\sin t)\,dt
+\int_{\sin t=\frac{1}{x}}^{\frac{\pi}{2}}\sin t \cos^2 t \frac{\pi}{2}dt)=
\frac{3}{2}\int_0^{\sin t=\frac{1}{x}}\frac{x\cos^4 t\, dt}{\sqrt{1-x^2\sin^2 t}},\\
\Im (p^{+}(x))=\frac{9}{2}\int_{\sin t=\frac{1}{x}}^{\frac{\pi}{2}}\sin t \cos^2 t \arccos (x\sin t)\,dt=\frac{3}{2}\int_{\sin t=\frac{1}{x}}^{\frac{\pi}{2}} \frac{x\cos^4 t\, dt}{\sqrt{x^2\sin^2 t-1}}.
\end{gather*}
Similar equalities hold for $x<-1$.
Furthermore
\begin{equation}\label{expforReIMp}
\Re (p^{+}(x))=\frac{3}{2}\int_0^{\frac{\pi}{2}} (1-x^{-2}\sin^2 u)^{\frac{3}{2}}\,du, \,\Im (p^{+}(x))=\frac{3}{2}(1-x^{-2})^2\int_0^{\frac{\pi}{2}} \frac{\sin^4 v}{\sqrt{1-(1-x^{-2})\sin^2 v}}\, dv
\end{equation}

\end{lemma}
\begin{proof} The arguments of the proof are the same as Haagerup plus the following modification.  First recall that $(\cos^3 t)'=-3\sin t\cos^2 t$.   Hence
\begin{gather*}
3\int_{0}^{\sin t=\frac{1}{x}}\sin t \cos^2 t \arcsin(x\sin t)dt=-\cos^3 t\arcsin(x\sin t)|_{t=0}^{\sin t=\frac{1}{x}}+\int_0^{\sin t=\frac{1}{x}}\frac{x\cos^4 t\, dt}{\sqrt{1-x^2\sin^2 t}},\\
3\int_{\sin t=\frac{1}{x}}^{\frac{\pi}{2}}\sin t \cos^2 t \frac{\pi}{2}dt=-\frac{\pi}{2}\cos^3 t|_{\sin t=\frac{1}{x}}^{\frac{\pi}{2}}.
\end{gather*}
These equalities show the first identity for $\Re (p^{+}(x))$.
Use the same integration by part for the first identity for $\Im (p^{+}(x))$.

 For identities  \eqref{expforReIMp} use the same substitutions as in \cite[page 205]{Haagerup}.  For the expression  $\Re (p^{+}(x))$ use $\sin u=x\sin t$.  Since in the integrant we have $\cos^4 t=\cos^2 t \cos^2 t$ we need to multiply the integrant of Haagerup by $\cos^2 t=1-\sin ^2 t=1-x^{-2}\sin^2 u$, which gives the factor $(1-x^{-2}\sin^2 u)^{\frac{3}{2}}$.
 
 For the expression  $\Im (p^{+}(x))$ use the substitution $\sin v=\frac{\cos t}{\sqrt{1-x^{-2}} }$.  Again $\cos^2 t=(1-x^{-2})\sin^2 v$.
\end{proof}
\begin{lemma}\label{expansp+} We have the following series expansions for $x\ge 1$: 
\begin{eqnarray}\label{expanspsi1}
\psi_1(x)=\Re (p^+(x))=\frac{9\pi}{16}\big( \frac{4}{3}-x^{-2}+\sum_{k=2}^{\infty} \frac{(2k-5)!!(2k-1)!!}{2^{2k-2}(k!)^2} x^{-2k}\big),\\
\label{expanspsi2}
\psi_2(x)=\Im (p^+(x))=\frac{3\pi}{16}\sum_{k=0}^{\infty}\frac{(2k-1)!!(2k+3)!!}{ 2^{2k}k!(k+2)! } (1-x^{-2})^{k+2}.
\end{eqnarray}
Furthermore the functions $\psi_1$ and $\psi_2$ strictly increase on for $x\ge 1$.
\end{lemma}

\begin{proof} First, we use the following Taylor expansions:
\begin{eqnarray*}
(1-t)^{\frac{3}{2}}=1-\frac{3}{2}t +3\sum_{k=2}^{\infty} \frac{(2k-5)!!}{2^k k!}t^k, \quad (1-t)^{-\frac{1}{2}}=\sum_{k=0}^{\infty}\frac{(2k-1)!!}{2^k k!} t^k.
\end{eqnarray*}
Second,  we use the formula $\int_0^{\frac{\pi}{2}} \sin^{2n} u du=\frac{(2n-1)!!}{2^{n+1}n!}\pi$, where $(-1)!!=1$.

To show that $\psi_1(x)$ strictly increases observe that $1-x^{-2}\sin^2 u$ strictly increasing for $x\ge 1$.   To show that $\psi_2(x)$  striclty increases for $x\ge 1$ observe that the functions $(1-x^{-2})$ and $\frac{1}{1-(1-x^{-2})\sin^2 u} $ strictly increase for $x\ge 1$.
\end{proof}

Use \eqref{defp+} and the arguments of \cite{Haagerup}[Lemmas 2.3 and  2.4] to deduce:
\begin{lemma}\label{lemma2.3H} 
\begin{enumerate}
\item $\Im (p^+(z))\ge \Im (p^+(|z|))$ for $|z|\ge 1$ , $\Im (z)\ge 0$.
\item $p^+(z)$ has no zero in $\mathbb{C}^+$ except $z=0$. 
\end{enumerate}
\end{lemma}
\begin{lemma}\label{lemma2.4H}  Assume that the Taylor series of $p^{-1}(x)$ are given by \eqref{invexpansionp}. Let $\alpha>1$.  Then for an odd positive integer $n$ we have 
\[
c_n =\frac{2}{\pi n}\int_{1}^{\alpha}\Im (p^{+}(x)^{-n})\,dx +r_n(\alpha), \textrm{ where } |r_n(\alpha)|\le \frac{\alpha}{n}(\Im p^+(\alpha))^{-n}.
\]
\end{lemma}

 We now  imitate the steps in the proof of Haagerup for nonpositivity of Taylor series of $h(z)$ for $k\ge 2$.  We will do that without relying on the complete elliptic integrals, only using Lemmas \ref{Halem2.2} and \ref{expansp+}.
 We first start with the following Lemma:
 \begin{lemma}\label{argincreas}  The ration $\frac{\psi_2(x)}{\psi_1(x)}$ strictly increases for $x\ge 1$.
 \end{lemma}
\begin{proof} Let
 \[\frac{\psi_2(x)}{\psi_1(x)}=\frac{(1-x^{-2})^2\int_0^{\frac{\pi}{2}} \frac{\sin^4 v}{\sqrt{1-(1-x^{-2})\sin^2 v}}}{\int_0^{\frac{\pi}{2}} (1-x^{-2}\sin^2 u)^{\frac{3}{2}}\,du}=
 \frac{(1-x^{-2})^{\frac{1}{2}}\int_0^{\frac{\pi}{2}} \frac{\sin^4 v}{\sqrt{1-(1-x^{-2})\sin^2 v}}}{\int_0^{\frac{\pi}{2}} [(1-x^{-2})^{-1}(1-x^{-2}\sin^2 u)]^{\frac{3}{2}}\,du}
 \]
 The proof of Lemma \ref{expansp+} yields that the numerator of the last expression strictly increases for $x\ge 1$.  Thus it is enough to show that the denominator of the last expression strictly decreases for $x\ge 1$.   This would follow from the claim that 
 \[\frac{1-(1-x^{-2})\sin^2 u}{1-x^{-2}}=\cos^2 u +\frac{1}{x^2-1}\]
 strictly decreasing.  This is obvious from the last last expression.
 \end{proof}
 Clearly 
 \[\psi_1(1)=1,\; \psi_1(\infty)=\frac{3\pi}{4}, \quad \psi_2(1)=0,\;\psi_2(\infty)=\infty.\]
 \begin{corollary}\label{thetaincres}  The complex function $p^{+}$ has the following expression for $x\ge 1$
 \[\theta(x)=\arctan \frac{\psi_2(x)}{\psi_1(x)},\, p^{+}(x)=\psi_1(x)+\mathrm{i}\psi_2(x)=|p^{+}(x)|\, e^{\mathrm{i}\theta(x)}=\sqrt{\psi_1^2 (x)+\psi^2_2(x)}\,e^{\mathrm{i}\theta(x)}.\]
 The function $\theta(x)$ strictly increases for $x\ge 1$, where $\theta(1)=0$ and $\theta(\infty)=\frac{\pi}{2}$.
 \end{corollary}
 The above corollary is the analog of \cite{Haagerup}[Lemma 2.7].
 
 Next we need an analog of Lemma 2.8. We first start with the following lemma .
 \begin{lemma}\label{funcchix}  Let $\chi(x): [0,\frac{\pi}{2})\to [1,\infty)$ be the inverse function of $\theta(x)$.  Then the substitution $x=\chi(y)$ yields:
 \begin{eqnarray*}&&\frac{d\chi}{dy}=\frac{\psi_1^2(\chi(y))+\psi_2^2(\chi(y))}{\psi_2'(\chi(y))\psi_1(\chi(y))-\psi_1'(\chi(y))\psi_2(\chi(y))},\\
 &&|p^{+}(x)|^{-n}dx=|p^{+}(\chi(y))|^{-n}\frac{\psi_1^2(\chi(y))+\psi_2^2(\chi(y))}{\psi_2'(\chi(y))\psi_1(\chi(y))-\psi_1'(\chi(y))\psi_2(\chi(y))}dy=\\
 &&\big(|p^{+}|^{n-2}(\chi(y))(\psi_2'(\chi(y))\psi_1(\chi(y))-\psi_1'(\chi(y))\psi_2(\chi(y)))\big)^{-1}.
 \end{eqnarray*}
 \end{lemma}
 \begin{proof}  As $y=\theta(x)$ is strictly increasing on $[1,\infty)$ it follows that $x=\chi(y)$ is strictly increasing on $[0,\frac{\pi}{2})$.  Clearly
 \[\frac{dy}{dx}=(\arctan\frac{\psi_2(x)}{\psi_1(x)})'=\frac{\psi_2'(x)\psi_1(x)-\psi_1'(x)\psi_2(x)}{\psi_1^2(x)+\psi_2^2(x)}.\]
 This equality implies straightforward the lemma.
 \end{proof}
 The following proposition follows from \eqref{expforReIMp}, the numerical calculations \ref{omegaincrlem} and the last part of Lemma \ref{expansp+}. 
 \begin{proposition}\label{decxpn}  Let $\omega(x)=\psi_2'(x)\psi_1(x)-\psi_1'(x)\psi_2(x)$.  
 Then $\omega(x)$ strictly increases on $[0, \tau]$, where 
 \[\tau \approx 1.732, \,\omega(\tau)\approx 1.360, \, \omega(1)=\omega'(\tau)=0.\]
 Assume that $\omega(x)|p^{+}(x)|^{m}$ is strictly increasing on $[1,\alpha]$ for some $m>0$.  Then for each integer $k\ge 0$ the function $\omega(x)|p^{+}(x)|^{m+k}$ strictly increases on the interval $[1,\alpha]$.
 \end{proposition}
\begin{figure}
\includegraphics[width=10cm]{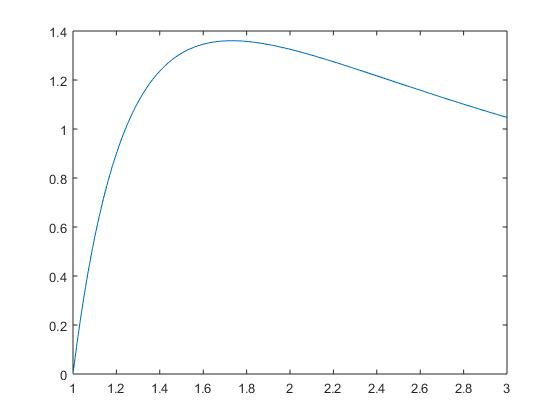}
\caption{Graph of function $\omega(x)$.}
\label{graph_omega}
\includegraphics[width=10cm]{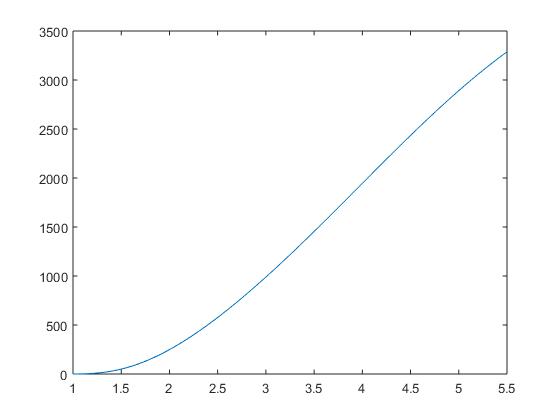}
\caption{Graph of function $\omega(x)|p^{+}(x)|^{7}$.}
\label{graph_omega_p}
\centering
\end{figure} 
Let us choose $\alpha=5$ and 
\[\theta_0=\theta(5)=\arctan\frac{\psi_2(5)}{\psi_1(5)}\approx 0.8097.\]
By Proposition~\ref{condompmincreas}, $\omega(x)|p^{+}(x)|^{7}$ increases in $[1,5]$. We now give an analog of \cite{Haagerup}[Lemma 2.8].

\begin{lemma}\label{HagL2.8}  Let $\alpha= 5$ and $\theta_0=\theta(\alpha)$.  For a fixed $n\in\mathbb{N}$ let $p=\lfloor\frac{n\theta_0}{\pi}\rfloor$.  Set
\[I_r=\frac{2}{\pi n}\int_{\theta(x)=\frac{\pi}{n}(r-1)}^{\theta(x)=\frac{\pi}{n}r} |p^{+}(x)|^{-n} |\sin n\theta(x)|dx\]
for $r=1,\dots,p$.  Put
\[I'=\frac{2}{\pi n}\int_{\theta(x)=\frac{\pi}{n}p}^{\alpha} |p^{+}(x)|^{-n} |\sin n\theta(x)|dx\]
Then:
\begin{enumerate}
\item \[\frac{2}{\pi n}\int_1^{\alpha}\Im (p^{+}(x)^{-n})dx=-I_1+I_2-\dots+(-1)^pI_p+(-1)^{p+1}I'.\]
\item For $n\ge 9$ one has $p\ge 2$ and $I_1>I_2>\dots>I_p>I'$.
\end{enumerate}
\begin{proof}(\emph{1}) Observe that $\Im(p^{+}(x)^{-n})=|p^{+}(x)|^{-n}\sin (-n\theta(x))=-|p^{+}(x)|^{-n}\sin (n\theta(x))$.  Hence
\[\int_1^{\alpha}\Im(p^{+}(x)^{-n})dx=-\int_1^{\alpha}|p^{+}(x)|^{-n}\sin (n\theta(x))dx=\frac{\pi n}{2}(-I_1+I_2-\dots+(-1)^pI_p+(-1)^{p+1}I').\]
(\emph{2})  Let $x=\chi(y)$ for $y\in[0,\theta_0]$.  Use Lemma \ref{funcchix} and the definition of $\omega(x)$ in Proposition \eqref{decxpn} to deduce
\begin{eqnarray}\label{formIr}
&&I_r=\frac{2}{\pi n}\int_{\frac{\pi}{n}(r-1)}^{\frac{\pi}{n}r} \big(\omega(\chi(y))|p^{+}(\chi(y))|^{n-2})^{-1}|\sin ny|dy,\\
&& I'=\frac{2}{\pi n}\int_{\frac{\pi}{n}p}^{\theta_0} \big(\omega(\chi(y))|p^{+}(\chi(y))|^{n-2}\big)^{-1}|\sin ny|dy.\notag
\end{eqnarray}
Recall that the function $\omega(x)|p^{+}(x)|^7$ strictly increases on $[1,5]$. Use Proposition \eqref{decxpn} to deduce that $\omega(x)|p^{+}(x)|^{n-2}$ is strictly increasing on $[1,\alpha]$ for $n\geq 9$.

Since $|\sin(ny)|$ is periodic with period $\pi/n$, it follows that 
\[
I_1>I_2>\dots>I_p.
\] 
Additionally, 
\begin{align*}
I^\prime &=\frac{2}{\pi n}\int_{\frac{\pi}{n}p}^{\theta_0} \big(\omega(\chi(y))|p^{+}(\chi(y))|^{n-2}\big)^{-1}|\sin ny|\, dy \\
&\leq \frac{2}{\pi n}\int_{\frac{\pi}{n}(p-1)}^{\theta_0-\pi/n} \big(\omega(\chi(y))|p^{+}(\chi(y))|^{n-2}\big)^{-1}|\sin ny|\, dy\\
&< I_p.
\end{align*}

\end{proof}
\end{lemma}
We now give the analog of $q$ in \cite{Haagerup}:
\begin{lemma}\label{defq}  Let 
\[
\mu(x)=\frac{\psi_2'(x)\psi_2(x)+\psi_1'(x)\psi_1(x)}{\psi_2'(x)\psi_1(x) -\psi_1'(x)\psi_2(x)}, \qquad x\in(1,2].
\]
 Then $\mu(x)$ strictly decreases on the interval $(1,1.732]$.  Furthermore
\[
(\log |p^{+}(\chi(y))|)'=\mu(\chi(y)),\qquad y\in(0,\frac{\pi}{2}).
\]
\end{lemma}
\begin{figure}
\includegraphics[width=10cm]{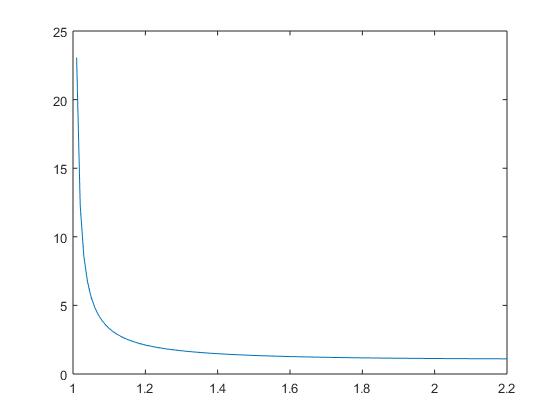}
\caption{Graph of function $\mu(x)$.}
\label{graph_mu}
\centering
\end{figure} 
\begin{proof} The claim that $\mu(x)$ strictly decreases on $(1,1.732]$ follows from the Proposition~\ref{mudecreasing}.  
Clearly
\[\log |p^{+}(x)|=\frac{1}{2} \log |p^{+}(x)|^2=\frac{1}{2}\log (\psi_2^2(x)+\psi_1^2(x)).\]
Hence
\[
(\log |p^{+}(\chi(y))|)'=\frac{\psi_2'(\chi(y))\psi_2(\chi(y))+\psi_1'(\chi(y))\psi_1(\chi(y))}{\psi_2^2(\chi(y))+\psi_1^2(\chi(y))}\chi'(y)=\mu(\chi(y)).
\]

\end{proof}

We now give the analog of \cite{Haagerup}[Lemma 2.9]:
\begin{lemma}\label{HagL2.9}   Let $n\ge 21$. $I_1,I_2,\dots$ are defined as in Lemma \ref{HagL2.8} and  $q=\mu(\tau)\approx 1.2020$.  Put 
\[c=|p^{+}(\tau)|e^{-q\theta(\tau)}\approx 1.2923.\]
Then
\begin{enumerate}
\item $I_1>\frac{0.326}{ n^2}c^{-n}$,
\item $I_2< 0.033 I_1$.
\end{enumerate}
\end{lemma}
\begin{proof}  Recall that $(\log |p^{+}(\chi(y))|)'=\mu(\chi(y))$ and $\mu(x)$ is strictly decreasing on $(1,1.732]$.
Hence $\mu(\chi(y))$ strictly decreasing on $(0,\theta(1.732)]$. In particular, for $y\in (0,\theta(\tau)]$
\[\mu(\chi(y))\ge \mu(\chi(\theta(\tau)))=\mu(\tau)=q\approx 1.2020.\]
Therefore
\begin{equation}\label{p+ineqint}
\log|p^{+}(\chi(u))|-\log|p^{+}(\chi(y))|=\int_{y}^{u}\mu(\chi(t))dt\ge q(u-y), \textrm{ for } 0\le y\le u\le\theta(\tau).
\end{equation}
Choose $u=\theta(\tau)$ to obtain
\begin{equation}\label{p+ineq}
|p^{+}(\chi(y))|\le ce^{qy} \textrm{ for } y\in [0,\theta(\tau)].
\end{equation}

(\emph{1})  
We first use \eqref{formIr} for $r=1$:
\[I_1=\frac{2}{\pi n}\int_{0}^{\frac{\pi}{n}} \big(\omega(\chi(y))|p^{+}(\chi(y))|^{n-2})^{-1}|\sin ny|dy  .\]
Since $n\ge 21$ it follows that 
\[\frac{\pi}{n}\le \frac{\pi}{21}\approx 0.1496<\theta(\tau)\approx 0.3224.\]
 As $\omega(x)$ is increasing on the interval $[1,\tau]$ it follows that $(\omega(\chi(y)))^{-1}> \omega^{-1}(\tau)$ on $(0, \frac{\pi}{n}]$.  Hence
 \[I_1>\frac{2}{\omega(\tau)\pi n}\int_{0}^{\frac{\pi}{n}} |p^{+}(\chi(y))|^{2-n} \sin ny \, dy  .\]
 Apply inequality \eqref{p+ineq} to deduce that
\begin{align*}
I_1 &> \frac{2}{\omega(\tau)\pi n} \int_0^{\frac{\pi}{n}}  (ce^{qy})^{2-n}\sin ny \,dy \\
&=\frac{2c^{2-n}}{\omega(\tau)\pi n^2} \int_0^{\pi} e^{-\frac{(n-2)qy}{n}}\sin y \,dy  \\ 
&\geq \frac{2c^{2-n}}{\omega(\tau)\pi n^2} \int_0^{\pi} e^{-qy}\sin y \,dy \\
&=\frac{2c^2}{\omega(\tau)\pi n^2}\frac{1+e^{-q\pi}}{1+q^2} c^{-n}.
\end{align*}
Since 
\[
\frac{2c^2}{\omega(\tau)\pi}\frac{1+e^{-q\pi}}{1+q^2}\approx 0.32697,
\]
this completes the proof of (1).

 (\emph{2}) We now use \eqref{formIr} for $r=2$:
\[I_2=\frac{2}{\pi n}\int_{\frac{\pi}{n}}^{\frac{2\pi}{n}} \big(\omega(\chi(y))|p^{+}(\chi(y))|^{n-2})^{-1}|\sin ny|dy  .\]
We now make a subsittution $y=t+\frac{\pi}{n}$ in the integral formula for $I_2$:
\[I_2=\frac{2}{\pi n}\int^{\frac{\pi}{n}}_0 \big(\omega(\chi(t+\frac{\pi}{n}))|p^{+}(\chi(t+\frac{\pi}{n}))|^{n-2})^{-1}\sin nt\,dt  .\]
Let
\[\hat I_2=\frac{2}{\pi n}\int^{\frac{\pi}{n}}_0 \big(\omega(\chi(t))|p^{+}(\chi(t+\frac{\pi}{n}))|^{n-2})^{-1}\sin nt\,dt.\]
Since $n\ge 21$ it follows that 
\[\frac{2\pi}{n}\le \frac{2\pi}{21}\approx 0.2992<\theta(\tau)\approx 0.3224.\]
As $\omega(x)$ is strictly increasing on the interval $[1,\tau]$, we know that
$(\omega(\chi(y)))^{-1}$ strictly decreases on $(0, \frac{2\pi}{n}]$. Then one gets $I_2<\hat I_2$.  Next use the inequality \eqref{p+ineqint} for $y=t$ and $u=t+\frac{\pi}{n}$ which yields
\[
e^{-\frac{(n-2)q\pi}{n}}|p^{+}(\chi(t))|^{-(n-2)}\ge|p^{+}(\chi(t+\frac{\pi}{n}))|^{-(n-2)}, \qquad t\in [0,\frac \pi n].
\]
Hence for $n\ge 21$
\[\hat I_2\le e^{-\frac{(n-2)q\pi}{n}} I_1\le e^{-\frac{19q\pi}{21}}I_1<0.033I_1.\]
\end{proof}

\begin{theorem}\label{nonpositpinvcorf}
$c_1=\frac {32}{9\pi}$ and $c_{2k+1}< 0$ for $k\geq1$. 
\end{theorem}
\begin{proof}
Let $k\geq 10$. Applying Lemma \ref{lemma2.4H} with $\alpha=5$ and Lemma \ref{HagL2.8}, we have 
\begin{align*}
-c_{2k+1}  &=I_1-I_2+\dots+(-1)^{p-1}I_p+(-1)^p I^\prime-r_{2k+1}(5)  \\
&> I_1-I_2-r_{2k+1}(5).
\end{align*}
Using Lemma \ref{HagL2.9}, we get 
\[
(I_1-I_2)>\frac{0.315}{(2k+1)^2}1.293^{-(2k+1)} \quad \text \quad |r_{2k+1}(5)|\leq \frac{5}{2k+1}2.4^{-(2k+1)}.
\]
Plugging above formulas into $-c_{2k+1}> I_1-I_2-r_{2k+1}(5)$, it follows that for $k\geq 10$
\begin{align*}
-c_{2k+1}&>(I_1-I_2)-r_{2k+1}(5) \\
&> \frac{0.315}{(2k+1)^2}1.293^{-(2k+1)}- \frac{5}{2k+1}2.4^{-(2k+1)} \\
&=\frac{0.315}{(2k+1)^2}1.293^{-(2k+1)}\left[1-\frac{5(2k+1)}{0.315}\left(\frac{1.293}{2.4}\right)^{2k+1} \right] \\
&>0.
\end{align*}

By directly using \eqref{inversion_formula}, we can obtain following approximations (rounded to two decimal places).

\begin{center}
\begin{tabular}{ c l c  l c l c l c }
 $n$     &1 & 3 & 5 & 7&9   \\     \hline 
\\ [-1em]
 $-c_n$  & $-32/9\pi$ &  0.12 &$ 4.84\cdot 10^{-3}$ & $2.58\cdot 10^{-3}$ & $1.22\cdot 10^{-3}$   \\   \hline
\\ [-1em]
 $n$     & 11&13&15&17 &19  \\     \hline
\\ [-1em]
 $-c_n$ &  $6.76\cdot 10^{-4}$ & $4.15\cdot 10^{-4}$ & $2.74\cdot 10^{-4}$ &  $1.91\cdot 10^{-4}$ & $1.39\cdot 10^{-4}$
\end{tabular}  
\end{center}
\end{proof}

Since $c_{2k+1}\leq0$ for $k\geq1$, $h_1(z)=c_1 z -h^{-1}(z)$ has nonnnegative Taylor coefficients. Vivanti-Pringsheim theorem yields that if the radius of convergence of Taylor series of $h_1(z)$ is $r$ then $h_1(z)$, and hence $h^{-1}(z)$, has a singular point at $r$.  As $h'(t)>0$ on $(0,1)$, and $h(1)=1$, it follows that $r\ge 1$ .  Clearly $h^{-1}(t)\le \sum_{k=0}^N c_{2k+1}t^{2k+1}$ for $t\in (0,1)$.  That is $\sum_{k=1}^N|c_{2k+1}|t^{2k+1}\le c_1t -h^{-1}(t)$ for $t\in (0,1)$. In particular $\sum_{k=1}^N|c_{2k+1}|\le c_1 -1$.  Thus $\sum_{k=0}^{\infty}|c_{2k+1}|\le 2c_1 -1$. 
 
As $c_1 \ne 0$, clearly, the function 
\begin{equation}\label{defpsix}
\psi(x) \coloneqq \sum_{k=0}^{\infty}|c_{2k+1}|x^{2k+1} 
\end{equation}
is a strictly increasing and continuous on $[0,r)$.  Recall that if the series 
\eqref{invexpansionp} converge for $x_0\in \mathbb{R}$ (pointwise) then $r\ge |x_0|$.
Hence $\psi(x)=+\infty$ for $x>r$.  $\psi(1)=\sum_{k=0}^{\infty}|c_{2k+1}|\geq c_1= 32/ (9\pi)>1$. Thus there exists a unique $c_0\in(0,1)$ such that $\psi(c_0)=1$.

 \begin{proposition}\label{existc0}  The following equality holds
 \[\sum_{k=1}^{\infty}|c_{2k+1}|=-\sum_{k=1}^{\infty}c_{2k+1}=c_1-p^{-1}(1)=\frac{32}{ 9\pi}-1.\]
 Hence the equation $\psi(x)=1$ has a unique solution $c_0\in (0, \frac{9\pi}{32})$ given by the equation $c_0=p(2c_1c_0-1)$.  Equivalently, let $x_0$ be the unique solution of 
 \begin{equation}\label{defx0}
 p(x_0)=\frac{9\pi(1+x_0)}{64}.
 \end{equation}
 Then
 \begin{equation}\label{alterdefc0}
 c_0=\frac{9\pi(1+x_0)}{64}.
 \end{equation}
 \end{proposition}
 \begin{proof}  Observe that $\psi(x)=2c_1x-p^{-1}(x)$.  Hence $\psi(c_0)=1$ is equivalent to $p^{-1}(c_0)=2c_1c_0-1$, which implies that $c_0=p(2c_1c_0-1)$.  Set
 $x_0=2c_1c_0-1$ and use $c_0=\frac{9\pi(1+x_0)}{64}$ to deduce the Proposition.
 \end{proof}
 
 \subsection{An upper bound on $K_G^{\mathbb{H}}$}
 \begin{theorem}\label{upbundKGH}   Let $x_0$ be the unique solution of \eqref{defx0}.  Then 
 \begin{equation}\label{upbundKGH1}
 K_G^{\mathbb{H}}\le \frac{64}{9\pi(x_0+1)}\approx 1.2168.
 \end{equation}
 \end{theorem}
The same value of an upper on  the constant $K_{G,4}^\mathbb{R}$: $1.216786$ was calculated in \cite[Table 1, p' 81]{Bri2}.  The authors explain on \S6, that the computation of Table 1 are ``just numerical, and do not yield a formal
proof''.

 \begin{proof}  Recall previous definition of $P(z)$ and $\psi(x)$
\[
 P^{-1}(z)=\arg z\, p^{-1}(|z|)=\sum_{k=0}^{\infty} c_{2k+1}z |z|^{2k}.
\]
\[
\psi(c_0) = \sum_{k=0}^{\infty}|c_{2k+1}|c_0^{2k+1} = 1. 
\]
Recall Lemma \eqref{kerneltrick}. Let $g(z)=\sum_{k=0}^\infty c_{2k+1} c_0^{2k+1}z^{2k}$.  Then $p^{-1}(c_0|z|)=|z|g(|z|)$.  Given unit vectors $x_1,\dots,x_m, y_1,\dots y_n$ in a quaternionic Hilbert space $\mathcal{H}$, then there exists unit vectors $u_1,\dots,u_m,v_1,\dots,v_n$ in a quaternionic Hilbert space $\mathcal{H}'$ such that $P^{-1}(c_0\langle x_i, y_j\rangle)=\langle u_i,v_j\rangle$.  Let $\mathcal{H}_1$ be an $l$-dimensional subspace of $\mathcal{H}'$ spanned by $u_1,\dots,u_m,v_1,\dots,v_n$.
Thus we can assume that $\mathcal{H}_1=\mathbb{H}^l$, where $l\le m+n$.
Assume that 
\noindent
$\max_{\lvert \varepsilon_i \rvert = \lvert \delta_j \rvert = 1}\biggl|\sum_{i=1}^m\sum_{j=1}^n M_{ij} \bar \varepsilon_i \delta_j\biggr|\leq 1$.
 Then
\[
\biggl|\sum_{i=1}^m\sum_{j=1}^n M_{ij}  \sign\langle u,{z}\rangle\sign\langle {z},v_j \rangle G_l^\mathbb{H}({z}) \biggr|\leq  G_l^\mathbb{H}({z}), \quad z\in \mathbb{H}^l.
\]
Integrate over the Lebesgue measure of $\mathbb{H}^l$ to get
$\biggl|\sum_{i=1}^m\sum_{j=1}^n M_{ij}  P (\langle u_i, v_j \rangle) \biggr|\leq  1$.
Now
\begin{eqnarray*}
1\ge \biggl| \sum_{i=1}^m\sum_{j=1}^n M_{ij} P (\langle u_i, v_j \rangle)\biggr|=
\biggl| \sum_{i=1}^m\sum_{j=1}^n M_{ij} c_0\langle  x_{i},y_{j}\rangle\biggr| .
\end{eqnarray*}
Take maximum on unit vectors $x_1,\dots,x_m,y_1,\dots,y_n$ to deduce that $\|M\|_G\le 1/c_0$.
\end{proof}

\subsection{The function $p(x)$ \eqref{taylorforp} and the constant $K_{G}^\mathbb{H}$}

We claim that the function $p(x)$ \eqref{taylorforp} is the function $\varphi_4^{\mathbb{R}}$  \cite[(40)]{FL20}.   The proof is very similar to part (ii) of Lemma 3.1 in \cite{FL20}.
Hence we have inequaities:
$K_{G,4}^\mathbb{R}\le K_{G,2}^\mathbb{C}\le K_{G}^\mathbb{H}$.

\section{Symmetric Versions of Grothendieck Inequality for quaternions}
In this section, we address the symmetric version of Grothendieck inequality for quaternions. Note that $\mathbb{S}^n(\mathbb{H})$ is the set of $n\times n$ quaternions matrices $A$ satisfying $A^*=A$ and $\mathbb{S}_+^n(\mathbb{H})$ is the set of positive semidefinite matrices $A$ satisfying $u^* A u \geq 0$ for all $u \in \mathbb{H}^n$.  There is no natural left or right action of quaternion scalar on the space $\mathbb{S}^n(\mathbb{H})$, so we should view it as a real vector space. Recall that on $\mathbb{S}^n(\mathbb{C})$ the inner product is given by $\langle A,B\rangle =\tr A^* B=\tr AB\in \mathbb{R}$.  Unfortunately, for $A,B\in \mathbb{S}^n(\mathbb{H})$, where $n\ge 2$ the $\tr A B$ does not to have to be real.  Since we view $\mathbb{S}^n(\mathbb{H})$ as a real vector space, we define an inner product on  $\mathbb{S}^n(\mathbb{F})$ as $\Re \tr AB=\frac{1}{2} \tr( AB +BA)$.  This definition is identical with the standard definition of the inner product for $\mathbb{F}\in \{\mathbb{R},\mathbb{C})$ and gives the right definition for quaternions.

We consider the quantities:
\[
\Re\sum_{i=1}^n \sum_{j=1}^na_{ij}\langle x_i, x_j \rangle , \quad \Re \sum_{i=1}^n \sum_{j=1}^n  a_{ij}\bar \delta_i\delta_j.
\] 
As in \S\ref{subsec: quatSDP} we can compute the maximum of the first term using SDP.  Also, as in \cite{FL20}, in this section we modify the definition of $\sign z$ for $z\in \mathbb{H}$:
\begin{eqnarray*}\label{eq:sign}
\sign z= \begin{cases}
z/\lvert z \rvert & z\ne 0,\\
1 & z = 0.
\end{cases}
\end{eqnarray*} 
This will yield the equality $|\sign z|=1$ for all $z$.  Clearly Theorem \ref{uvfuvfor} will still apply for this definition of $\sign z$.
\subsection{Symmetric Grothendieck inequality}
Denote by $\mathbb{S}^n_o(\mathbb{H})$
the real subspace of all quaternion self-adjoint matrices with zero diagonal.  Then $A\in \mathbb{S}^n_o(\mathbb{H})$ is of the form $A=D+A_0$ where $A_0=(a_{ij,0})\in\mathbb{S}^n_o(\mathbb{H})$ and $D$ is a real diagonal matrix.
So $\tr A=\tr D$. Observe that 
\begin{eqnarray*}
 \sum_{i=1}^n \sum_{j=1}^n a_{ij}\langle x_i, x_j \rangle=\tr A +\sum_{i=1}^n \sum_{j=1}^n a_{ij,0}\langle x_i, x_j \rangle \textrm { if } \|x_i\|=1.
\end{eqnarray*}
For $A\in \mathbb{S}^n(\mathbb{H})$ we consider the following quantities, which are  analogous to the quantities we introduced in \cite{FL20}:
\begin{eqnarray*}
&\|A\|_\theta= \max _{|\delta_i|=1}\biggl| \Re\sum_{i=1}^n \sum_{j=1}^n  a_{ij} \bar\delta_i \delta_j  \biggr|,\quad
&\|A\|_\Theta= \max _{|\delta_i|\le1}\biggl| \Re\sum_{i=1}^n \sum_{j=1}^n  a_{ij}\bar \delta_i\delta_j  \biggr|,\\
&\|A\|_\gamma=\max _{\|x_i\|=1}\biggl|\Re \sum_{i=1}^n \sum_{j=1}^n a_{ij}\langle x_i, x_j \rangle\biggr| , \quad &\|A\|_\Gamma=\max _{\|x_i\|\le1}\biggl|\Re \sum_{i=1}^n \sum_{j=1}^n a_{ij}\langle x_i, x_j \rangle\biggr|.
\end{eqnarray*}
Clearly $\|A\|_\theta\le \|A\|_\gamma$ and $\|A\|_\Theta\le \|A\|_\Gamma$.
As in \cite{FL20} we observe that $\|A\|_\theta=0$ if and only if $A$ is a a diagonal matrix, and $\tr A=0$.    Indeed, observe that
\begin{eqnarray*}
\sum_{\delta_i\in\{-1,1\}} \sum_{i=1}^n\sum_{j=1}^n a_{ij} \bar\delta_i\delta_j=2^n\tr A.
\end{eqnarray*}
Assume that $\|A\|_\theta=0$.  Then we obtain that $\tr A=0$.  Hence $\|A\|_\theta=\|A_0\|_\theta=0$.  
Observe next that 
\begin{eqnarray*}
\sum_{\delta_i\in\{-1,1\}, i\ge 3} \sum_{i=1}^n\sum_{j=1}^n a_{ij}=2^{n-2} (a_{12}\bar\delta_1\delta_2 +a_{21} \bar\delta_2\delta_1)=2^{n-2} (a_{12}\bar\delta_1\delta_2 +\bar a_{12} \bar\delta_2\delta_1).
\end{eqnarray*}
Suppose that $a_{12}\ne 0$.   Choose $\delta_1=\sign a_{12}$  and $\delta_2=1$.
Then $(a_{12}\bar\delta_1\delta_2 +\bar a_{12} \bar\delta_2\delta_1)=2|a_{12}|$ which is a real number.  Hence the assumption that $\|A_0\|_\theta=0$ yields that $a_{12}=0$.  Similarly we deduce that $A_0=0$.  Hence $\|A\|_\gamma=0$ if and only if $A$ is a diagonal matrix with zero trace. Denote by $\mathbb{S}^n_=(\mathbb{H})$ the real subspace of self-adjoint matrices $A=(a_{ij})$ where $a_{11}=\dots=a_{nn}$.
Then $\|A\|_\theta$ and  $\|A\|_\gamma$ are norms on  $\mathbb{S}^n_=(\mathbb{H})$.

We now claim that $\|A\|_\Gamma$ is a norm on $\mathbb{S}^n(\mathbb{H})$.  Clearly, $\|A\|_\Gamma$ is a seminorm.  Choose $\delta_i=1$ and $\delta_j=0$ for $i\ne j$.  Then $|\Re\sum_{i=1}^n\sum_{j=1}^n a_{ij}\bar\delta_i\delta_j|=|a_{ii}|=0$.  Hence $A=A_0$.  As $\|A_0\|_\theta\le \|A\|_\Theta=0$ we deduce that $A_0=0$.  Thus $A=0$ and $\|\cdot\|_\Theta$ is a norm.  As $\|A\|_\Theta\le \|A\|_\Gamma$ it follows that $\|\cdot\|_\Gamma$ is a norm on   $\mathbb{S}^n(\mathbb{H})$.

Let $\mathcal{D}^n\subset \mathbb{S}^n_+(\mathbb{R})$ be the convex subset of all positive semidefinite diagonal matrices whose diagonal entries are in $[0,1]$.
As in \cite{FL20} it is straightforward to show that 
\begin{eqnarray}\label{charTGnorms}
\|A\|_\Theta=\max\{\|DAD\|_\theta, D\in \mathcal{D}^n\}, \quad \|A\|_\Gamma=\max\{\|DAD\|_\gamma, D\in \mathcal{D}^n\}.
\end{eqnarray}
Let $K_\gamma^\mathbb{H},K_\Gamma^\mathbb{H}$ the smallest possible constant,
(which in principle may equal to $\infty$), for which on has the inequalities
\begin{equation}\label{eq:SGI}
\lVert A \rVert_\gamma \le K_\gamma^\mathbb{H} \lVert A \rVert_\theta, \qquad \lVert A \rVert_\Gamma \le K_\Gamma^\mathbb{H} \lVert A \rVert_\Theta.
\end{equation}
\begin{theorem}[Symmetric Grothendieck inequality]\label{symgro}
The symmetric Grothendieck constants satisfy the following relations
\begin{eqnarray}\label{basGHin}
K_G^\mathbb{H}\le K_\Gamma^\mathbb{H} \leq K_\gamma^\mathbb{H} \leq \frac{64}{9\pi}-1\approx 1.263537.
\end{eqnarray}
\end{theorem}

\begin{proof}  The first inequality follows from the proof of Lemma \ref{SDPGHchar}.
The second inequality follows from the characterization \eqref{charTGnorms}.  We now show the third inequality.
Assume that $\|A\|_\theta\le 1$: $\biggl| \Re\sum_{i=1}^n \sum_{j=1}^n a_{ij}\bar \delta_i \delta_j  \biggr| \le 1$ for $|\delta_i|=1$.  
Let $x_1,\dots,x_n$ be unit vectors in a right Hilbert space $\mathcal{H}$.  Hence the span of $x_1,\dots,x_n$ is contained in a subspace of dimension at most $n$.  Thus we shall assume that $x_1,\dots,x_n\in \mathbb{H}^n$.
We claim that 
\begin{eqnarray*}
\biggl| \Re\sum_{i=1}^n \sum_{j=1}^na_{ij} P(\langle x_i, x_j \rangle) |P(\langle x_i, x_j \rangle) |^{2k}\biggr|\le 1, \quad k+1\in\mathbb{N}.
\end{eqnarray*}
Let
\begin{eqnarray*}
\Phi(x,z_1,\dots,z_{2k+1})=\sign \langle z_1,x\rangle \sign \langle x,z_2\rangle \sign \langle z_3,x\rangle \dots   \sign  \langle  z_{2k+1},x\rangle, x, z_1,\dots, z_{2k+1}\in \mathbb{H}^n.
\end{eqnarray*}
We claim that 
\begin{eqnarray*}
\int_{(\mathbb{H}^n)^{2k+1}} \overline{\Phi(x,z_1,\dots,z_{2k+1})} \Phi(y,z_1,\dots,z_{2k+1})  \prod_{p=1}^{2k+1} G^\mathbb{H}(z_p)dm(z_p)=P(\langle x,y\rangle)|P(\langle x,y\rangle)|^{2k}.
\end{eqnarray*}
Assume first that $\langle x,y\rangle$ is a real number.  Then the above equality follows by applying Theorem \ref{uvfuvfor} $2k+1$ times.  Observe next that for $a\in \mathbb{H}, |a|=1$ one has the equality $\Phi(xa,z_1,\dots,z_{2k+1})= \Phi(x,z_1,\dots,z_{2k+1})a$.  Hence by replacing $x$ with $x \sign \langle x,y\rangle$ we deduce the above equality.

Clearly
\begin{eqnarray*}
\biggl| \Re\sum_{i=1}^n \sum_{j=1}^n a_{ij}\overline{\Phi(x_i,z_1,\dots,z_{2k+1})} \Phi(x_j,z_1,\dots,z_{2k+1}) \biggr| \le 1.
\end{eqnarray*}
Hence
\begin{eqnarray*}
&&\biggl| \Re\sum_{i=1}^n \sum_{j=1}^na_{ij} P(\langle x_i, x_j \rangle) |P(\langle x_i, x_j \rangle) |^{2k}\biggr|=\\
&&\biggl| \Re\sum_{i=1}^n \sum_{j=1}^n a_{ij}(\int_{(\mathbb{H}^n)^{2k+1}} \overline{\Phi(x_i,z_1,\dots,z_{2k+1})} \Phi(x_j,z_1,\dots,z_{2k+1}) \prod_{p=1}^{2k+1}G^\mathbb{H}(z_p)dm(z_p))\biggr|\le \\
&&\int_{(\mathbb{H}^n)^{2k+1}}\biggl| \Re\sum_{i=1}^n \sum_{j=1}^n a_{ij} \overline{\Phi(x_i,z_1,\dots,z_{2k+1})} \Phi(x_j,z_1,\dots,z_{2k+1}) \biggr|\prod_{p=1}^{2k+1}G^\mathbb{H}(z_p)dm(z_p)\le 1.
\end{eqnarray*}

Finally
\begin{eqnarray*}
&&\biggl| \Re\sum_{i=1}^n \sum_{j=1}^n a_{ij}\langle x_i, x_j \rangle \biggr| = \biggl| \Re\sum_{i=1}^n \sum_{j=1}^n a_{ij}P^{-1}(P(\langle x_i, x_j \rangle)) \biggr|\\
 &&\leq  \sum_{k=0}^\infty |c_{2k+1}| \biggl| \Re\sum_{i=1}^n \sum_{j=1}^na_{ij} P(\langle x_i, x_j \rangle) |P(\langle x_i, x_j \rangle) |^{2k}\biggr|\le \sum_{k=0}^\infty |c_{2k+1}|=\frac{64}{9\pi}-1.
\end{eqnarray*}
\end{proof}
\subsection{Cones of positive semidefinite matrices}
Assume that $A\in\mathbb{S}^n(\mathbb{R}), B\in\mathbb{S}^n(\mathbb{H})$. Then
$\tr AB=\tr BA$.  Hence $\Re\tr AB=\tr AB$.
For real positive definite matrix, the quantities we are considering have a special relation:
\begin{lemma}\label{psdlemma}
For $A\in \mathbb{S}_+^n(\mathbb{R})$,
\[
\max _{\|x_i\|=1}\biggl| \sum_{i=1}^n \sum_{j=1}^n a_{ij}\langle x_i, x_j \rangle \biggr| = \max _{\|x_i\|\leq1}\biggl| \sum_{i=1}^n \sum_{j=1}^n a_{ij}\langle x_i, x_j \rangle \biggr| = \max _{\|x_i\|=1, \|y_j\| = 1}\biggl| \sum_{i=1}^n \sum_{j=1}^n a_{ij}\langle x_i, y_j \rangle \biggr|
\]
where $x_1,\dots ,x_n, y_1,\dots, y_n\in \mathbb{H}^l$ and
\[
\max _{|\delta_i|=1}\biggl| \sum_{i=1}^n \sum_{j=1}^n a_{ij}\langle \delta_i, \delta_j \rangle \biggr| = \max _{|\delta_i|\leq1}\biggl| \sum_{i=1}^n \sum_{j=1}^n a_{ij}\langle \delta_i, \delta_j \rangle \biggr| = \max _{|\delta_i|=1, |\epsilon_j|=1}\biggl| \sum_{i=1}^n \sum_{j=1}^n a_{ij}\langle \delta_i, \epsilon_j \rangle \biggr|
\]
where $\delta_1,\dots ,\delta_n, \epsilon_1, \dots, \epsilon_n \in \mathbb{H}$.
\end{lemma}
\begin{proof}
Observe that by definition
\[
\max _{\|x_i\|=1}\biggl| \sum_{i=1}^n \sum_{j=1}^n a_{ij}\langle x_i, x_j \rangle \biggr| \leq \max _{\|x_i\|\leq1}\biggl| \sum_{i=1}^n \sum_{j=1}^n a_{ij}\langle x_i, x_j \rangle \biggr| \leq \max _{\|x_i\|=1, \|y_j\| = 1}\biggl| \sum_{i=1}^n \sum_{j=1}^n a_{ij}\langle x_i, y_j \rangle \biggr|.
\]
Define the matrices $X = [x_1, \dots,x_n], Y =  [y_1, \dots,y_n]$. Because $A$ is real positive definite. Then $A=B^2$, where $B=(b_{ij})=A^{1/2}$.
\[
\langle X, Y \rangle = \sum_{i=1}^n \sum_{j=1}^n a_{ij}\langle x_i, y_j \rangle=\sum_{p=1}^n \langle \sum_{i=1}^n x_i b_{ip}, \sum_{j=1}^n y_j b_{jp}\rangle.
\]
is an inner product of the space $\mathbb{H}^{l\times n}$. By Cauchy-Schwarz, we have
\[
|\langle X, Y \rangle | \leq \sqrt{\langle X, X \rangle \langle Y, Y \rangle}  \leq \max _{\|x_i\|=1}\biggl| \sum_{i=1}^n \sum_{j=1}^n a_{ij}\langle x_i, x_j \rangle \biggr|.
\]
So
\[
\max _{\|x_i\|=1, \|y_j\| = 1}\biggl| \sum_{i=1}^n \sum_{j=1}^n a_{ij}\langle x_i, y_j \rangle \biggr| \leq \max _{\|x_i\|=1}\biggl| \sum_{i=1}^n \sum_{j=1}^n a_{ij}\langle x_i, x_j \rangle \biggr|
\]
and the first equality holds. By taking $l = 1$ in the previous argument, we see that the second equality holds.
\end{proof}

The next theorem is a generalization of the Nesterov $\pi/2$-Theorem.
\begin{theorem}\label{symgro}
Let $A = (a_{ij})\in \mathbb{S}^n_+(\mathbb{R})$ is a symmetric positive semidefinite matrices. Then
\begin{equation}
\max _{\|x_i\|=1}\biggl| \sum_{i=1}^n \sum_{j=1}^n a_{ij}\langle x_i, x_j \rangle \biggr| \leq \frac{32}{9\pi} \max _{|\delta_i|=1}\biggl| \sum_{i=1}^n \sum_{j=1}^n a_{ij}\langle \delta_i, \delta_j \rangle \biggr|.
\end{equation}
where $x_1,\dots ,x_n\in \mathbb{H}^l$, $\delta_1,\dots ,\delta_n\in \mathbb{H}$. By Lemma~\ref{psdlemma}, we can change the expression in both sides of the inequality.  
\end{theorem}

\begin{proof}
Assume 
\[
\max _{|\delta_i|=1}\biggl| \sum_{i=1}^n \sum_{j=1}^n a_{ij}\langle \delta_i, \delta_j \rangle \biggr| = 1.
\]
Then for any $x_1,\dots ,x_n\in \mathbb{H}^l$,
\[
\sum_{i=1}^n \sum_{j=1}^n a_{ij} P(\langle x_i,x_j\rangle) \leq 1.
\]
Consider the matrix $G_{ij} = \langle x_i,x_j\rangle$. The first coefficient of $P$ is $\frac{9\pi}{32}$. By the Schur product theorem for quaternions,
\[
A\circ P(G) - A \circ \frac{9\pi}{32}G \succeq 0.
\]
So 
\[
\sum_{i=1}^n \sum_{j=1}^n a_{ij} P(\langle x_i,x_j\rangle) \geq \frac{9\pi}{32} \sum_{i=1}^n \sum_{j=1}^n a_{ij} \langle x_i,x_j\rangle.
\]
So the inequality holds.
\end{proof}

The constant $\frac{9\pi}{32}$ is sharp. In fact, the proof of Theorem 5.1 in \cite{FL20} can be repeated with only minor difference.

\subsection{Cones of weighted Laplacians}
For any matrix $A = (a_{ij})\in \mathbb{S}^n(\mathbb{R})$ with zero diagonal and positive off-diagonal elements, we can define $L_{A}$ by
\[
L_A:= \diag(A \mathbbm{1})-A,
\]
where $\diag(x)$ denotes the diagonal matrix whose diagonal is $x$ and $\mathbbm{1}$ is the vector of all ones. And we can define the set of all weighted Laplacians:
\[
\mathbb{L}^n := \{ L_A : A\in \mathbb{S}^n(\mathbb{R}), a_{ii} = 0, a_{ij}\geq0 \text{ for all }i, j\} = \{L\in \mathbb{S}^n(\mathbb{R}): L\mathbbm{1}=0, l_{ij} \leq 0\text{ for all }i\neq j\}.
\]
We have $\mathbb{L}^n \subseteq \mathbb{S}_+^n(\mathbb{R})$.

\begin{theorem}
Define the quaternion Goemans-Williamson constant:
\[
\alpha_{GW}^{\mathbb{H}}:= \inf_{0\leq x \leq 1}\frac{1+P(x)}{1+x}.
\]

Then for all $L\in \mathbb{L}^n$,
\begin{equation}
\max _{\|x_i\|=1}\biggl| \sum_{i=1}^n \sum_{j=1}^n  l_{ij}\langle x_i, x_j \rangle \biggr| \leq \frac{1}{\alpha_{GW}^{\mathbb{H}}} \max _{|\delta_i|=1}\biggl| \sum_{i=1}^n \sum_{j=1}^n l_{ij}\langle \delta_i, \delta_j \rangle \biggr|.
\end{equation}
where $x_1,\dots ,x_n\in \mathbb{H}^l$, $\delta_1,\dots ,\delta_n\in \mathbb{H}$. The value of the constant is approximate $0.967337$, so $K\leq 1.0338$. By Lemma~\ref{psdlemma}, we can change the expression in both sides of the inequality. 

\end{theorem}
\begin{proof}
We will show that for $h\in \mathbb{H}$, we have
\[
\alpha_{GW}^{\mathbb{H}}= \inf_{0\leq x \leq 1}\frac{1+P(x)}{1+x} = \inf_{|h| < 1}\frac{1-\Re[ P(h)]}{1-\Re(h)}.
\]

By definition,
\[
\inf_{|h| < 1} \frac{1-\Re[ P(h)]}{1-\Re(h)} \leq \inf_{0\leq x \leq 1} \frac{1+P(x)}{1+x}.
\]
On the other hand, if $\Re(h) \geq 0$, let $h = x+y$, $x = \Re(h)$. The Taylor expansion of $P(x)$ is given in formula~\ref{taylorforp}, $P(x) = \sum_{i=0} b_{2i+1} x^{2i+1}$. $b_{2i+1} \geq 0$ and $P(1) = 1$.
\[
\Re[P(h)] = \sum_{i=0} b_{2i+1} x (x^2 + |y|^2)^i \leq \sum_{i=0} b_{2i+1} x = x.
\]
So
\[
\inf_{|h| < 1, \Re(h) \geq 0} \frac{1-\Re[ P(h)]}{1-\Re(h)} \geq 1 \geq \inf_{0\leq x \leq 1}\frac{1+P(x)}{1+x}.
\]

If $\Re(h) \leq 0$, let $h = -x+y$, $x = -\Re(h)$. Then
\[
1-\Re[P(h)] = 1+ \sum _{i=0} b_{2i+1} x (x^2 + |y|^2)^i \geq 1+\sum_{i=0} b_{2i+1} x^{2i+1} = 1+P(x).
\]
So
\[
\inf_{|h| < 1, \Re(h) \leq 0} \frac{1-\Re[ P(h)]}{1-\Re(h)} \geq \inf_{0\leq x \leq 1} \frac{1+P(x)}{1+x}.
\]

So we have 
\[
\inf_{|h| < 1} \frac{1-\Re[ P(h)]}{1-\Re(h)} \geq \inf_{0\leq x \leq 1} \frac{1+P(x)}{1+x}
\]
and the equality holds.

Let $L_A \in \mathbb{L}^n$. Assume
\[
\max _{|\delta_i|=1}\biggl| \sum_{i=1}^n \sum_{j=1}^n l_{ij}\langle \delta_i, \delta_j \rangle \biggr| = 1.
\]

For any unit vectors $x_1,\dots ,x_n\in \mathbb{H}^l$,
\[
\int_{z\in \mathbb{H}^l} \sum_{i=1}^n\sum_{j=1}^n l_{ij} \sign \langle x_i, z \rangle \sign \langle z, x_j \rangle G_n(z) dm(z) \leq 1.
\]

Because $L$ is a real symmetric matrix, we have
\begin{align*}
& \ \ \ \int_{z\in \mathbb{H}^l} \sum_{i=1}^n\sum_{j=1}^n l_{ij} \sign \langle x_i, z \rangle \sign \langle z, x_j \rangle G_n(z) dm(z)  \\
&= \sum_{i=1}^n\sum_{j=1}^n l_{ij} P(\langle x_i, x_j \rangle)\\
&= \sum_{i=1}^n\sum_{j=1}^n l_{ij} \Re[ P(\langle x_i, x_j \rangle)]\\
&= \sum_{i=1}^n\sum_{j=1}^n a_{ij} (1-\Re [P(\langle x_i, x_j \rangle)])\\
&\geq \alpha_{GW}^{\mathbb{H}} \sum_{i=1}^n\sum_{j=1}^n a_{ij} (1-\Re [\langle x_i, x_j \rangle])\\
&= \alpha_{GW}^{\mathbb{H}} \sum_{i=1}^n\sum_{j=1}^n a_{ij} (1-\langle x_i, x_j \rangle)\\
&=\alpha_{GW}^{\mathbb{H}} \sum_{i=1}^n\sum_{j=1}^n l_{ij} \langle x_i, x_j \rangle.
\end{align*}
Hence
\[
\max _{\|x_i\|=1}\biggl| \sum_{i=1}^n \sum_{j=1}^n  l_{ij}\langle x_i, x_j \rangle \biggr| \leq \frac{1}{\alpha_{GW}^{\mathbb{H}}}.
\]
\end{proof}

\subsection{Cones of diagonally dominant matrices}
As in last section, all matrices considered in this section will be real. Let $\mathbb{S}^n_{dd}:= \{A\in \mathbb{S}^n(\mathbb{R}): a_{ii} \geq \sum_{i\neq j}|a_{ij}| \}$ be the cone of symmetric diagonally dominant matrices. Let $\mathbb{S}^n_{dd}(\mathbb{R}_+)$ be the cubcone of diagonally dominant matrices with nonnegative entries.
\begin{lemma}
Every $A \in \mathbb{S}^n_{dd}$ has a unique decomposition $A = P+L$ such that $P\in \mathbb{S}^n_{dd}(\mathbb{R}_+), L \in \mathbb{L}^n$ and $l_{ij}p_{ij} = 0$ whenever $i\neq j$.
\end{lemma}
\begin{proof}
Let $B$ be defined by $b_{ij} = -a_{ij}$ if $a_{ij}<0$ and $i\neq j$, $b_{ij} = 0$ otherwise. Then $B$ has zero diagonal and positive off-diagonal elements and $L_B \in \mathbb{L}^n$. Let $P:= A-L_B$. Then $p_{ij} \geq 0$ and $l_{ij}p_{ij} = 0$ whenever $i\neq j$. Since $a_{ii} \geq \sum_{i\neq j}|a_{ij}| = \sum_{i\neq j} (l_{ij}+p_{ij})$, we have $p_{ii} = a_{ii} - \sum_{i\neq j} l_{ij} \geq \sum_{i\neq j} p_{ij}$, so $P\in \mathbb{S}^n_{dd}$. Uniqueness follows by the definition.
\end{proof}

Clearly,
\[
\mathbb{L}^n \subseteq \mathbb{S}^n_{dd} \subseteq \mathbb{S}^n_{+}(\mathbb{R}).
\]
By definition, we expect the constant of Geothendieck inequality for $\mathbb{S}^n_{dd}$ lies between those of $\mathbb{L}^n$ and $\mathbb{S}^n_{+}(\mathbb{R})$. The following theorem gave a bound for the constant.

\begin{theorem}
Let $a_0 = \frac{9\pi}{32}$ be the constant given in Theorem~\ref{symgro} and $\alpha_{GW}^{\mathbb{H}}$ be the quaternion Goemans-Williamson constant. Then for any $A \in \mathbb{S}^n_{dd}$,
\begin{equation}
\max _{\|x_i\|=1}\biggl| \sum_{i=1}^n \sum_{j=1}^n  a_{ij}\langle x_i, x_j \rangle \biggr| \leq \bigg(1+\frac{1-a_0}{\alpha_{GW}^{\mathbb{H}}} \bigg) \max _{|\delta_i|=1}\biggl| \sum_{i=1}^n \sum_{j=1}^n a_{ij}\langle \delta_i, \delta_j \rangle \biggr|.
\end{equation}
where $x_1,\dots ,x_n\in \mathbb{H}^l$, $\delta_1,\dots ,\delta_n\in \mathbb{H}$. The value of the constant is approximate $1.1204$. By Lemma~\ref{psdlemma}, we can change the expression in both sides of the inequality. 
\end{theorem}
\begin{proof}
Let $A = P+L$ be the decomposition given by the lemma. Assume that
\[
\max_{|\delta_i|=1}\biggl| \sum_{i=1}^n \sum_{j=1}^n a_{ij}\langle \delta_i, \delta_j \rangle \biggr| = 1.
\]
We have
\begin{align*}
1&\geq \sum_{i=1}^n \sum_{j=1}^n a_{ij}P(\langle x_i, x_j \rangle) \\
&= \sum_{i=1}^n \sum_{j=1}^n l_{ij}P(\langle x_i, x_j \rangle) + \sum_{i=1}^n \sum_{j=1}^n p_{ij}P(\langle x_i, x_j \rangle)\\
&\geq \alpha_{GW}^{\mathbb{H}} \sum_{j=1}^n  l_{ij}\langle x_i, x_j \rangle+ a_0 \sum_{i=1}^n \sum_{j=1}^n  p_{ij}\langle x_i, x_j \rangle\\
&=\alpha_{GW}^{\mathbb{H}} \sum_{j=1}^n  a_{ij}\langle x_i, x_j \rangle - (\alpha_{GW}^{\mathbb{H}}  - a_0)\sum_{j=1}^n  p_{ij}\langle x_i, x_j \rangle.
\end{align*}
Since $P\in \mathbb{S}^n_{dd}(\mathbb{R}_+)$, we have
\[
\sum_{i=1}^n \sum_{j=1}^n p_{ij}\langle x_i, x_j \rangle \leq \sum_{i=1}^n \sum_{j=1}^n p_{ij} = \sum_{i=1}^n \sum_{j=1}^n a_{ij} \leq 1
\]
So
\[
\alpha_{GW}^{\mathbb{H}} \sum_{j=1}^n  a_{ij}\langle x_i, x_j \rangle \leq 1+\alpha_{GW}^{\mathbb{H}}  - a_0
\]
and the inequality follows.
\end{proof}

\newpage 
\begin{appendix}
\section{Additional results on quaternions}
\subsection{Representation of quaternions as real $4\times 4$ matrices}
By representing a complex number $z=x+y\mathrm{i}$ as $2\times 2$ real valued matrix $\left(\begin{array}{cc}x&y\\-y&x\end{array}\right)$ we can represent $C(a)$
as 
\[C(a,\mathbb{R})=\left(\begin{array}{rrrr}x&y&u&v\\-y&x&-v&u\\-u&v&x&-y\\-v&-u&y&x\end{array}\right), \quad a=(z,w), z=x+y\mathrm{i},w=u+v\mathrm{i} .\]
Then $a\to C(a,\mathbb{R})$ is an isomorphism of $\mathbb{H}$ and the induced  $4$-dimensional subalgebra $\mathcal{C}(\mathbb{H},\mathbb{R})=\{C(a,\mathbb{R}), \;a\in\mathbb{H}\}\subset\mathbb{R}^{4\times 4}$.

\subsection{Elementary operations and rank}
One can perform elementary row operations on $M$ to bring $M$ to row echelon (REF) or reduced row echelon form (RREF).  The number of pivots is called the row rank of $M$.  The row rank of $M$ is the dimension of the row space of $M$. 
This elementary row operations can be implemented by the corresponding product of the elementary matrices from the left hand side of $M$.  We can also perform the corresponding elementary column operations, when we multiply the columns be the scalar from the right.   This can be done by the corresponding product of the elementary matrices from the right hand side of $M$.  
It is equivalent to finding to REF or RREF of $A^{\top}$.  (See the arguments below.) The row rank of $A^{\top}$ will give us the dimension of the columns space generated by the columns of $A$, viewed as the right vector space over $\mathbb{H}$.  We claim that the row rank and the column rank are equal.  As $C(a)$ is invertible if $a\ne 0$ and $C(0)=0$, we deduce that
the  elementary row operations on $M$ corresponding to the elementary block row operations on $C(M)$.  Hence the row rank of $M$ is a half of the rank of $C(M)$.  Elementary column operations on $M$ correspond to the elementary block column operations on $C(M)$.  Hence the column rank of $M$ is a half of the rank of $C(M)$. 

\subsection{Additional results in quaternion inner product space}
In what follows we need the following lemma
 \begin{lemma}\label{nprodbariden} Let $x,y\in\mathbb{H}^l$.  Then
 \begin{equation}\label{inprodbariden1}
 \Re(\langle x,y\rangle)=\Re(\langle\overline{y},\overline{x}\rangle).
 \end{equation}
 \end{lemma}
 \begin{proof}  Use \eqref{tracecomrel} and the definitions of $\overline{x},\overline{y}$ to deduce
 \[\Re(\langle x,y\rangle)=\Re(\sum_{i=1}^l x_i \bar y_i)=\Re(\sum_{i=1}^l \bar y_i x_i )=\Re(\langle \overline{y},\overline{x}\rangle).\]
 
 \end{proof}
 
Let $b_1,\dots,b_m$ be an orthonormal basis in a right vector space $\mathbb{V}$.
Then
\begin{equation}\label{orthexpanid}
x=\sum_{i=1}^m b_i\langle b_i,x\rangle,  \,
y=\sum_{i=1}^m b_i\langle b_i,y\rangle\,
\langle x,y\rangle=\sum_{i=1}^n \langle x,b_i\rangle \langle b_i, y\rangle.
\end{equation}
\begin{definition}\label{defbarxB}  Let $\mathbb{V}$ be a right $m$-dimensional vector space over $\mathbb{H}$ with the inner product $\langle \cdot \rangle$.  Let $\mathcal{B}=\{b_1,\dots,b_m\}$ be an orthonormal basis in $\mathbb{V}$.  For a vector $v\in \mathbb{V}$ we define 
\[\overline{v}_{\mathcal{B}}=\sum_{i=1}^m b_i \overline{\langle b_i,v\rangle} =\sum_{i=1}^m b_i\langle v, b_i\rangle .\]
\end{definition}

\section{Increasing properties of three functions}
Recall the definiton of our functions:
\begin{align}\label{expforReIMp1}
\psi_1(x)&=\Re (p^{+}(x))=\frac{3}{2}\int_0^{\frac{\pi}{2}} (1-x^{-2}\sin^2 u)^{\frac{3}{2}}\,du, \\
\psi_2(x)&=\Im (p^{+}(x))=\frac{3}{2}(1-x^{-2})^2\int_0^{\frac{\pi}{2}} \frac{\sin^4 v}{\sqrt{1-(1-x^{-2})\sin^2 v}}\, dv, \label{expforReIMp2}\\
\omega(x) &= \psi_2'(x) \psi_1(x) - \psi_2(x) \psi_1'(x),\\
\label{specvalpsi12} \psi_1(1)&=1, \quad \psi_1(\infty)=\frac{3\pi}{4}, \quad \psi_2(1)=0, \quad \psi_2(\infty)=\infty
\end{align}
It is easy to show from the definitions that $\psi_1(x)$ and $\psi_2(x)$ are increasing functions on $[1,\infty)$.
We are interested in some properties of related functions on the interval $[1,5]$.  For the function $\psi_2(x)$, since it is a function of $y=\sqrt{1-x^{-2}}$ for $x\in[1,\infty)$,
$y\in[0, \sqrt{24}/5]$,
we can differentiate this function as many times as needed on the interval $[1,5]$.  However, for the function $\psi_1(x)$, one can show  that the second derivative does not exists at $1$, i.e., it value at $1$ is $\infty$.

We also observed in \cite{FLZ} that  $\frac{\psi_2(x)}{\psi_1(x)}$ strictly increases on $[1,\infty)$, see Lemma 2.17.  Furthermore,
\begin{align}\label{expanspsi1}
\psi_1(x)=\Re (p^+(x))=\frac{9\pi}{16}\big( \frac{4}{3}-x^{-2}+\sum_{k=2}^{\infty} \frac{(2k-5)!!(2k-1)!!}{2^{2k-2}(k!)^2} x^{-2k}\big),\\
\label{expanspsi2}
\psi_2(x)=\Im (p^+(x))=\frac{3\pi}{16}\sum_{k=0}^{\infty}\frac{(2k-1)!!(2k+3)!!}{ 2^{2k}k!(k+2)! } (1-x^{-2})^{k+2}
\end{align}

\subsection{Basic properties}
 Set 
 \begin{align}\label{phi1form}
 \phi_1= \frac{3}{4}\big(\frac{4}{3}-x^{-2}+\sum_{k=2}^{\infty} \frac{(2k-5)!!(2k-1)!!}{2^{2k-2}(k!)^2} x^{-2k}\big),\\
 \label{phi3form}
 \phi_2=\sum_{k=0}^{\infty}\frac{(2k-1)!!(2k+3)!!}{ 2^{2k}k!(k+2)! } (1-x^{-2})^{k+2}.
 \end{align}
 Thus
 \begin{align*}
 \psi_1=\frac{3\pi}{4}\phi_1, \;\psi_2(x)=\frac{3\pi}{16}\phi_2, \;\omega=\frac{9\pi^2}{64}\tilde\omega, \;\tilde\omega=\phi_2'\phi_1 -\phi_2\phi_1'.
 \end{align*}
 Furthermore
 \begin{align}\label{phi1'form}
 \phi_1'&= \frac{3}{2}\big(x^{-3}-\sum_{k=2}^{\infty} \frac{(2k-5)!!(2k-1)!!}{2^{2k-2}(k!)(k-1)!} x^{-2k-1}\big),\\
 \label{phi1'form}
 \phi_1''&= \frac{3}{2}\big(-3x^{-4}+\sum_{k=2}^{\infty} \frac{(2k-5)!!(2k+1)!!}{2^{2k-2}(k!)(k-1)!} x^{-2k-2}\big),\\
  \label{phi2'form}
 \phi_2'&=2x^{-3}\sum_{k=2}^{\infty}\frac{(2k-1)!!(2k+3)!!}{ 2^{2k}k!(k+1)! } (1-x^{-2})^{k+1},\\
  \label{phi2''form}
 \phi_2''&=2x^{-6}\sum_{k=2}^{\infty}\frac{(2k-1)!!(2k+3)!!}{ 2^{2k}k!(k+1)! } (1-x^{-2})^{k}(2k+5-3x^2).
 \end{align}

We introduce following finite series, where $m$ will be speccified later. All those functions are rational functions.
\begin{align*}
 \phi_{1,m}(x) &= \frac{3}{4}\big(\frac{4}{3}-x^{-2}+\sum_{k=2}^{m} \frac{(2k-5)!!(2k-1)!!}{2^{2k-2}(k!)^2} x^{-2k}\big),\\
 \hat \phi_{1,m}(x) &= \frac{3}{4}\big(\frac{4}{3}-x^{-2}+(\sum_{k=2}^{m-1} \frac{(2k-5)!!(2k-1)!!}{2^{2k-2}(k!)^2} x^{-2k})\\ 
	&+\frac{(2m-5)!!(2m-1)!!}{2^{2m-2}(m!)^2} x^{-2(m-1)}\frac{1}{x^2-1}\big),\\
	\phi_{1,m}'(x) &=\frac{3}{2}\big(x^{-6}-\sum_{k=2}^{m} \frac{(2k-5)!!(2k-1)!!}{2^{2k-2}k!(k-1)!} x^{-2k-1}\big),\\
	\bar\phi_{1,m}(x) &=\frac{3}{2}\big(x^{-6}-(\sum_{k=2}^{m-1} \frac{(2k-5)!!(2k-1)!!}{2^{2k-2}k!(k-1)!} x^{-2k-1})\\
	&-\frac{(2m-5)!!(2m-1)!!}{2^{2m-2}m!(m-1)!} x^{-2m+1}\frac{1}{x^2-1}\big),\\
   \phi_{1,m}''(x) &= \frac{3}{2}\big(-3x^{-4}+\sum_{k=2}^{m} \frac{(2k-5)!!(2k+1)!!}{2^{2k-2}k!(k-1)!} x^{-2k-2}\big),\\
   \tilde\phi_{1,m}(x) &= \frac{3}{2}\big(-3x^{-4}+\sum_{k=2}^{m-1} \frac{(2k-5)!!(2k+1)!!}{2^{2k-2}k!(k-1)!} x^{-2k-2}\\
   &+\frac{(2m-5)!!(2m+1)!!}{2^{2m-2}m!(m-1)!} x^{-2m}\frac{1}{x^2-1}\big).\\
   \phi_{2,m}(x)&=\sum_{k=0}^{m}\frac{(2k-1)!!(2k+3)!!}{ 2^{2k}k!(k+2)! } (1-x^{-2})^{k+2},\\
    \hat\phi_{2,m}(x)&=\sum_{k=0}^{m-1}\frac{(2k-1)!!(2k+3)!!}{ 2^{2k}k!(k+2)! } (1-x^{-2})^{k+2}\\
&+\frac{(2m-1)!!(2m+3)!!}{ 2^{2m}m!(m+2)! } (1-x^{-2})^{m+2} x^2,\\
    \phi_{2,m}'(x)&=2x^{-3}\sum_{k=0}^{m}\frac{(2k-1)!!(2k+3)!!}{ 2^{2k}k!(k+1)! } (1-x^{-2})^{k+1},\\ 
\bar\phi_{2,m}(x) &=2x^{-3}\sum_{k=0}^{m-1}\frac{(2k-1)!!(2k+3)!!}{ 2^{2k}k!(k+1)! } (1-x^{-2})^{k+1}\\
&+\frac{(2m-1)!!(2m+3)!!}{ 2^{2m}m!(m+1)! } (1-x^{-2})^{m+1} x^2,\\
 \phi_{2,m}''(x)&=2x^{-6}\sum_{k=0}^{m}\frac{(2k-1)!!(2k+3)!!}{ 2^{2k}k!(k+1)! } (1-x^{-2})^{k}\big(2k+5-3x^2\big).
 \end{align*}
 We claim that we have the following inequaities:
\begin{proposition}\label{basinvarphi}
For $ x > 1$,
 \begin{align*}\label{inphi12m}
 &0 < \phi_{1,m}(x)<\phi_1(x)< \hat \phi_{1,m}(x),\\
 &\bar\phi_{1,m}(x)<\phi_1'(x) <\phi_{1,m}'(x),0<\phi_{1}'(x),\\
 &\phi_{1,m}''(x)<\phi_{1}''(x)<\tilde\phi_{1,m}(x),\\
 &0< \phi_{2,m}(x)<\phi_2(x)< \hat \phi_{2,m}(x),\\
 &0 <{\phi_{2,m}'(x)} < \phi_2'(x) < \bar\phi_{2,m}(x).\notag
 \end{align*}
When  $m\geq 35, 1\leq x \leq 5$,
\[
\phi_{2,m}''(x) < \phi_2''(x)
\]
\end{proposition}
\begin{proof}
All the above inequalities are clear except the five inequalities 
 \begin{align*}
 &\phi_1(x)< \hat \phi_{1,m}(x),\;\bar\phi_{1,m}(x)<\phi_1'(x), \;\phi_1''(x)<\tilde \phi_{1,m}(x),\\
 &\phi_2(x)< \hat \phi_{2,m}(x), \phi_2'(x) < \bar\phi_{2,m}(x).  
 \end{align*}
 To show  the inequality $\phi_1(x)< \hat \phi_{1,m}(x)$, we argue as follows.  First observe that coefficients of $\phi_1(x)$ is strictly decreasing for $k\ge 2$.  This follows from the observation that the ratio between two terms 
\[
\frac{(2k-3)(2k+1)}{4(k+1)^2}<1
\]
for $k\ge 2$.  Now in the infinite series of $\phi_1(x)$ we replace the coefficient $ \frac{(2k-5)!!(2k-1)!!}{2^{2k-2}(k!)^2}$ by $ \frac{(2k-m)!!(2m-1)!!}{2^{2m-2}(k!)^2}$ for $k\ge m$.  This will increase the value of $\phi_1(x)$.  The infinite sum for $k\ge m$ can be summed to 
\[\frac{(2m-5)!!(2m-1)!!}{2^{2m-2}(m!)^2} x^{-2m} \frac{1}{1-x^{-2}}=
 \frac{(2m-5)!!(2m-1)!!}{2^{2m-2}(m!)^2} x^{-2(m-1)} \frac{1}{x^2-1}.
\]
The other inequalities can be shown similarly.
\end{proof}

To prove certain poperties of those function, we use Mathematica. Essentially, we only use \texttt{CountRoots} to calculate the number of roots of a rational function in a given interval. If the number is 0, we can conclude that the function stays positive or negative in this interval. \texttt{CountRoots} applies an exact algorithm so that we can prove our results rigorously. 

To bound infinite series by finite series, we need this trivial lemma:
\begin{lemma}\label{triviallemma}
For two real numbers $a,b$, if $0\le a_1 \le a \le a_2$, $b_1\le b\le b_2$, then $\min\{a_1 b_1, a_2 b_1\} \leq ab \leq \max\{a_1b_2, a_2b_2\}$. If furthermore $b\ge 0$, then $a_1 b_1 \leq ab \leq a_2 b_2$.
\end{lemma}
\begin{proof} Consider $\min ab$,  where $a,b$ are variables that satisfy $a_1\le a\le a_2$ and $b_1\le b\le b_2$.  Then this minimum is $\min\{a_ib_j, i,j\in[2]\}$.  We now use the assumption that $0\le a_1$.  Hence $a_1b_1\le a_1b_2, a_2b_1\le a_2b_2$.
Hence $ab\ge \min\{a_1b_1,a_2b_1\}$.  Similarly
\begin{eqnarray*}
 \max ab=\max\{a_ib_j, i,j\in[2]\}=\max\{a_1b_2, a_2b_2\}.
\end{eqnarray*}
Assume now that $b_2\ge b\ge 0$. Then 
\begin{eqnarray*}
\min\{ab, a\in[a_1,a_2]\}=\min\{a_1b,a_2b\}=a_1b, \quad \max\{ab, a\in[a_1,a_2]\}=\max\{a_1b,a_2b\}=a_2b\le a_2b_2.
\end{eqnarray*}
\end{proof}

\subsection{The function $\omega(x)$}
In this section, $m = 50$. 
\begin{proposition}\label{omegaincrlem}  Assume that $\tau\in(1,2)$ is the smallest value of $x \geq 1$ such that $\tilde\omega'(x)=0$.  Then $\tau > 1.732$.
\end{proposition}
 \begin{proof}
\[
\tilde\omega'(x) = \phi_2''(x)\phi_1(x)-\phi_2(x)\phi_1''(x).
\]
We now apply Lemma \ref{triviallemma} as follows.  First set
\begin{eqnarray*}
a_1=\phi_{1,m}(x), a=\phi_1(x), a_2=\hat \phi_{1,m}(x), b_1=\phi_{2,m}''(x)<\phi_2''(x), \quad x\in [1,5], m=50,
\end{eqnarray*}
to deduce
\begin{eqnarray*}
\phi_2''(x)\phi_1(x) \geq \min\{ \phi_{2,m}''(x)\phi_{1,m}(x), \phi_{2,m}''(x) \hat\phi_{1,m}(x) \}.
\end{eqnarray*}
Second set
\begin{eqnarray*}
a_1=\phi_{2,m}(x), a=\phi_2(x),a_2=\hat\phi_{2,m}(x), b_2=\tilde\phi_{1,m}(x), \quad x\in[1,5], m=50,
\end{eqnarray*}
to deduce
\begin{eqnarray*}
\phi_2(x)\phi_1''(x) \leq \max\{ \phi_{2,m}(x)\tilde\phi_{1,m}(x), \hat\phi_{2,m}(x)\tilde\phi_{1,m}(x)  \},
\end{eqnarray*}
which is equivalent to
\begin{eqnarray*}
-\phi_2(x)\phi_1''(x) \geq \min\{ -\phi_{2,m}(x)\tilde\phi_{1,m}(x),- \hat\phi_{2,m}(x)\tilde\phi_{1,m}(x)  \}.
\end{eqnarray*}

So $\tilde\omega'(x)$ is larger than the minimum of $2\times 2 = 4$ functions. Using \texttt{CountRoots}, we found that those 4 functions are all positive in $[1, 1.732]$, thus $\tau > 1.732$.
 \end{proof}

\subsection{The function $\omega(x)p_+^l$}
In this section, $m = 50$. Let
\begin{eqnarray}\label{defp+}
p_+(x)=\sqrt{\psi_1^2(x)+\psi_2^2(x)}=\frac{3\pi}{16}\sqrt{\tilde p}, \quad \tilde p(x)=16\phi_1^2+\phi_2^2.
\end{eqnarray}

\begin{proposition}\label{condompmincreas}
The function $\omega(x)p_+^l$ increases in the interval $[1,5]$ for $l=7$.
\end{proposition}
\begin{proof}
Note that since $\omega$ is increasing in the interval $[1,\tau]$ we automatically have that for all $l\ge 1$ the function $\omega(x)p_+^l$ increases on $[1,\tau]$. So now we have to verified that $\omega(x)p_+^l$ increases on $[1.732, 5]$.

Set
  $s_l(x)=\tilde \omega(x) \tilde p^{l/2}(x)$.  Then
  \begin{equation}\label{s'form}
  s_l'=\tilde p^{(l-2)/2}\big((l/2)\tilde\omega\tilde p' + \tilde \omega'\tilde p \big).
  \end{equation}
Let
\[
\rho =(l/2)\tilde\omega\tilde p' + \tilde \omega'\tilde p =l\tilde\omega (16\phi_1 \phi_1' +\phi_2 \phi_2') +  \tilde \omega' (16\phi_1^2+\phi_2^2)
\]
We need to prove $\rho \geq 0 $ in $[\tau,5]$. When $x$ is in $[1.732,5]$, by \texttt{CountRoots}, we have $\bar\phi_{1,m}>0, \tilde\phi_{1,m}<0, \phi_{2,m}''<0$. 
In view of Proposition \ref{basinvarphi} we deduce that 
\begin{eqnarray*}
0<\bar\phi_{1,m}(x)<\phi_1'(x), \quad 0<-\tilde\phi_{1,m}(x)<-\phi_1''(x)
 \textrm{ for } x\in[1.732,5].
\end{eqnarray*}
Furthermore, Proposition \ref{basinvarphi} yields
$\tilde\omega = \phi_2'\phi_1 -\phi_1'\phi_2 \geq \phi_{2,m}'\phi_{1,m} - \phi_{1,m}'\hat \phi_{2,m}$. Again, when $x$ in $[1.732,5]$, by \texttt{CountRoots}, $\phi_{2,m}'\phi_{1,m} - \phi_{1,m}'\hat \phi_{2,m}$ is positive. So we can bound from below the whole term $(l/2)\tilde\omega\tilde p'$ by $l(\phi_{2,m}'\phi_{1,m} - \phi_{1,m}'\hat \phi_{2,m})(16\phi_{1,m} \bar\phi_{1,m} +\phi_{2,m} \phi_{2,m}') $. This term is positive.

For the second term, if $\tilde \omega' \ge 0$, we have nothing to prove. So we can assume $\tilde \omega' <0$. Because $\tilde\omega' = \phi_2''\phi_1 -\phi_2\phi_1''$, and $-\phi_2\phi_1''$ is positive, so we can assume $\phi_2''$ is negative. 
Recall that  for $m=50$ and $x\in [1,5]$ we have the inequality $\phi_{2,m}''(x)<\phi_2''(x)$.  Hence for $x\in [1.732,5]$ we have the inequality $\phi_{2.m}''(x)<\phi_2''(x)<0$.  Therefore
\begin{eqnarray*}
\tilde \omega' (16\phi_1^2+\phi_2^2)>(\phi_{2,m}''\hat\phi_{1,m} -\phi_{2,m}\tilde\phi_{1,m}) (16\hat\phi_{1,m}^2+\hat\phi_{2,m}^2), \quad x\in[1.732,5], m=50.
\end{eqnarray*}
Finally, use \texttt{CountRoots} to test whether $l(\phi_{2,m}'\phi_{1,m} - \phi_{1,m}'\hat \phi_{2,m})(16\phi_{1,m} \bar\phi_{1,m} +\phi_{2,m} \phi_{2,m}') + (\phi_{2,m}''\hat\phi_{1,m} -\phi_{2,m}\tilde\phi_{1,m}) (16\hat\phi_{1,m}^2+\hat\phi_{2,m}^2)$ is positive in $[1.732,5]$. The answer is yes.
\end{proof}
  
\subsection{The function $\mu(x)$}\label{nufunc}
  In this section, $m = 40$. Recall the function 
  \[
\mu(x)=\frac{\psi_2'(x)\psi_2(x)+\psi_1'(x)\psi_1(x)}{\psi_2'(x)\psi_1(x) -\psi_1'(x)\psi_2(x)}.
\]

Note that
\begin{eqnarray*}
\frac{1}{4\mu(x)}=\frac{\phi_2'(x)\phi_1(x) -\phi_1'(x)\phi_2(x)}{\phi_2'(x)\phi_2(x)+16\phi_1'(x)\phi_1(x)}.
\end{eqnarray*}

Define $\nu(x)$ as
\begin{align*}
\big(\frac{1}{4\mu(x)}\big)'&=\frac{\nu(x)}{(\phi_2'(x)\phi_2(x)+16\phi_1'(x)\phi_1(x))^2},\\
\nu(x)&=(\phi_2'(x)\phi_1(x) -\phi_1'(x)\phi_2(x))'(\phi_2'(x)\phi_2(x)+16\phi_1'(x)\phi_1(x))-\\
&(\phi_2'(x)\phi_1(x) -\phi_1'(x)\phi_2(x))(\phi_2'(x)\phi_2(x)+16\phi_1'(x)\phi_1(x))'.
\end{align*}
Thus
\begin{align*}
\nu=\phi_2''\phi_1'(16\phi_1^2+\phi_2^2) -\phi_2'\phi_1''(16\phi_1^2+\phi_2^2) -\tilde\omega (16\phi_1'^2+\phi_2'^2).
\end{align*}

\begin{proposition}\label{mudecreasing}
The function $\nu(x)$ is positive in $[1, 1.732]$, so $\mu(x)$ is decreasing in this interval.
\end{proposition}
\begin{proof}
We have to consider two different intervals: $[1.01, 1.732]$ and $[1, 1.01]$.

Assume that $x\in[1.01, 1.732]$. By \texttt{CountRoots}, $\bar\phi_{1,m}(x)>0$ in $[1.01, 1.732]$, so $\phi_1'(x)>\bar\phi_{1,m}(x)>0$ in $[1.01, 1.732]$. Recall Proposition \ref{basinvarphi} and Lemma \ref{triviallemma}, for $x\in [1.01,1.732]$ and $m=40$:
\begin{align*}
a_1&=\bar\phi_{1,m}(x)(16\phi_{1,m}^2(x)+\phi_{2,m}^2(x)), a=\phi_1'(x)(16\phi_1^2(x)+\phi_2^2(x)),\\ a_2&=\phi_{1,m}'(x)(16\hat\phi_{1,m}^2(x)+\hat\phi_{2,m}^2(x)),\\
b_1&=\phi_{2,m}''(x), b=\phi_2''(x), b_2=\bar\phi_{2,m}(x).
\end{align*}
Then we have 
\[
\phi_2''\phi_1'(16\phi_1^2+\phi_2^2) \geq \min\{ \phi_{2,m}''\bar\phi_{1,m}(16\phi_{1,m}^2+\phi_{2,m}^2), \phi_{2,m}''\phi_{1,m}'(16\hat\phi_{1,m}^2+\hat\phi_{2,m}^2) \}.
\]

For the second term, $-\phi_2'\phi_1''(16\phi_1^2+\phi_2^2)$, by Proposition \ref{basinvarphi}:
\begin{align*}
a_1&=\phi_{2,m}'(x)(16\phi_{1,m}^2(x)+\phi_{2,m}^2(x)), a=\phi_2'(x)(16\phi_1^2(x)+\phi_2^2(x)),\\ a_2&=\bar\phi_{2,m}(x)(16\hat\phi_{1,m}^2(x)+\hat\phi_{2,m}^2(x)),\\
b_1&=\phi_{1,m}''(x),b=\phi_1''(x), b_2=\tilde\phi_{1,m}(x).
\end{align*}
Use the inequality $ab\le \max\{a_1b_2, a_2 b_2\}$ in Lemma \ref{triviallemma} to deduce
\[
-\phi_2'\phi_1''(16\phi_1^2+\phi_2^2)\geq \min\{ -\bar \phi_{2,m}\tilde \phi_{1,m}(16\hat\phi_{1,m}^2+\hat\phi_{2,m}^2), -\phi_{2,m}'\tilde \phi_{1,m}(16\phi_{1,m}^2+\phi_{2,m}^2) \}.
\]

For the last term $-\tilde\omega (16\phi_1'^2+\phi_2'^2)$ we proceed as follows:  First recall that $\tilde \omega>0$ for $x>1$.  Use Proposition \ref{basinvarphi} and Lemma \ref{triviallemma} to deduce $\tilde\omega_2(x)<\bar\phi_{2,m}\hat\phi_{1,m}(x)-\phi_{2,m}\bar\phi_{1,m}(x)$ for $x>1$. Hence 
\begin{eqnarray*}
\tilde \omega(x)(16\phi_1'^2(x)+\phi_2'^2(x))<(\bar\phi_{2,m}\hat\phi_{1,m}(x)-\phi_{2,m}\bar\phi_{1,m}(x))((\phi_{1,m}'(x))^2+(\bar\phi_{2,m}(x))^2)\Rightarrow\\
-\tilde \omega(x)(16\phi_1'^2(x)+\phi_2'^2(x))>-(\bar\phi_{2,m}\hat\phi_{1,m}(x)-\phi_{2,m}\bar\phi_{1,m}(x))((\phi_{1,m}'(x))^2+(\bar\phi_{2,m}(x))^2)
\end{eqnarray*}

There are $2\times 2 =4$ cases, in each case, their sum is positive in $[1.01, 1.732]$. So $\nu(x)$ is positive when $x\in[1.01, 1.732]$.

Assume that $x\in[1, 1.01]$. For $\phi_1'(x)$, we need a better lower bound than $\bar\phi_{1,m}(x)$ because when $x \to 1$, $\bar\phi_{1,m}(x)$ diverges to $-\infty$ while $\phi_1'(x)$ remains finite.

Recall the definition:
\[
\phi_1(x) = \frac{4}{3 \pi}\psi_1(x)=\frac{2}{\pi}\int_0^{\frac{\pi}{2}} (1-x^{-2}\sin^2 u)^{\frac{3}{2}}\,du.
\]
Take the derivative and note that $\pi < 22/7$,
\begin{align*}
\phi_1(x) &=\frac{6}{\pi}x^{-3} \int_0^{\frac{\pi}{2}}\sin^2 u (1-x^{-2}\sin^2 u)^{\frac{1}{2}}\,du \\
& \geq \frac{6}{\pi}x^{-3} \int_0^{\frac{\pi}{2}}\sin^2 u (1-\sin^2 u)^{\frac{1}{2}}\,du \\
&=  \frac{6}{\pi}x^{-3} \int_0^{\frac{\pi}{2}}\sin^2 u \cos u\,du = \frac{2}{\pi x^3} \geq \frac{7}{11 x^3}
\end{align*}

For the first term, Proposition  \ref{basinvarphi} yields
 $\phi_2''(x) \geq \phi_{2,m}''(x), \phi_1''(x) \geq \phi_{1,m}''(x) $. By \texttt{CountRoots}, $\phi_{2,m}''(x) >0, \phi_{1,m}''(x) >0$.  Hence 
\[ \phi_2''\phi_1'(16\phi_1^2+\phi_2^2) \geq  \phi_{2,m}''\frac{7}{11 x^{3}}(16\phi_{1,m}^2+\phi_{2,m}^2).
\]

For the second term, because $\phi_1, \phi_2$ are increasing functions, we can bound them above by value at $1.01$. So the whole term can be bounded below using Proposition \ref{basinvarphi}  by $-\bar \phi_{2,m}\tilde \phi_{1,m}(16\hat\phi_{1,m}(1.01)^2+\hat\phi_{2,m}(1.01)^2)$. 

For the last term, we use the same bound as for the interval $[1.01, 1.732]$. Finaly, the sum of three terms is positive in $[1, 1.01]$ using \texttt{CountRoots}. So $\nu(x)$ is positive when $x\in[1, 1.01]$.
\end{proof}

\end{appendix}
\section*{Acknowledgment}
The work of the first author is partially supported by the Simons Collaboration Grant for Mathematicians.
The work of the second author is partially supported by by the Neubauer Family Distinguished Doctoral Fellowship from the University of Chicago.
The work of the third author is partially supported by NSF IIS 1546413 and the Eckhardt Faculty Fund.

\end{document}